\documentclass[11pt]{article}

\usepackage[T1]{fontenc}
\usepackage{lmodern}
\usepackage[margin=1in]{geometry}
\usepackage{amsmath,amssymb,amsthm,mathtools}
\usepackage{booktabs,array,longtable,enumitem,microtype,graphicx}
\usepackage{tikz}
\usetikzlibrary{arrows.meta,positioning,calc}
\usepackage[hidelinks]{hyperref}

\usepackage{url}
\usepackage{placeins}
\newtheorem{theorem}{Theorem}[section]
\newtheorem{lemma}[theorem]{Lemma}
\newtheorem{proposition}[theorem]{Proposition}
\newtheorem{corollary}[theorem]{Corollary}
\newtheorem{definition}[theorem]{Definition}
\newtheorem{question}{Question}[section]
\theoremstyle{remark}
\newtheorem{remark}[theorem]{Remark}
\newtheorem{example}[theorem]{Example}
\newcommand{\Hb}{\bar H_w}
\newcommand{\Hrb}{\bar H_r}
\DeclareMathOperator*{\argmax}{arg\,max}
\setlist{nosep}

\newcommand{\quantifierbox}[1]{\par\medskip\noindent\fbox{\begin{minipage}{\dimexpr\linewidth-2\fboxsep-2\fboxrule\relax}\textbf{Quantifiers.} #1\end{minipage}}\par\medskip}
\newcommand{\paperTitle}{Inverse knapsack at two capacities: which pairs of value--cardinality hulls are realizable?}
\hypersetup{pdftitle={Inverse knapsack at two capacities: which pairs of value\textendash{}cardinality hulls are realizable?},pdfauthor={Prashant Chaudhary; Kapil Khandelwal},pdfkeywords={inverse optimization, 0-1 knapsack, cardinality-constrained knapsack, value functions, exchange arguments}}

\title{\paperTitle}
\author{%
  Prashant Chaudhary\thanks{Independent Researcher, San Francisco Bay Area, CA, USA.
    ORCID: \href{https://orcid.org/0009-0001-1980-109X}{0009-0001-1980-109X}.
    Corresponding author.}
  \and
  Kapil Khandelwal\thanks{Department of Civil and Environmental Engineering and Earth Sciences,
    University of Notre Dame, Notre Dame, IN, USA.
    ORCID: \href{https://orcid.org/0000-0002-5748-6019}{0000-0002-5748-6019}.}%
}
\date{}
\begin{document}
\maketitle
\begin{abstract}
One item set evaluated at two capacities $R<D$ produces two concave hulls of optimal value against cardinality. We ask which prescribed pairs arise. Exchange arguments give a necessary system on vertex witnesses, exchange closure (EC), whose scalar consequences form the linear closure. We exhibit a pair satisfying every scalar test that fails EC, so the linear closure is strictly larger already at larger terminal count three; and a globally coherent EC witness admitting no common-size representation although its target pair is realizable. Under a cardinality cap, a four-band classification of one family gives exact thresholds for cap-four realizability, uncapped realizability and the vertex-only capped closure. Pairs whose larger terminal count is at most two are characterized. Under the explicit encoding the decision problem lies in $\Sigma_2^p$ and is polynomial-time for fixed terminal cardinalities. Exact finite certificates establish agreement of the scalar and witness conditions on six specified domains; sufficiency of uncapped EC remains open.
\end{abstract}
\section{Introduction}\label{sec:intro}
Which pairs of optimal-value-versus-cardinality hulls can arise from a \emph{single} set of item values and sizes evaluated at two capacities? Concavity and dominance of one hull over the other are necessary, but they do not express the restrictions imposed by common item sizes. We study this inverse existence question rather than the forward task of optimizing a given instance. We call the capacity-$D$ problem the \emph{larger} problem, indexed by $w$, and the capacity-$R$ problem the \emph{remainder} problem, indexed by $r$.

The organizing distinction is between existence of a target pair and completion of a \emph{specified} witness. The linear closure $L$ retains scalar inequalities; exchange closure (EC) retains labeled vertex attainments and exchange constraints; full witness closure (FWC) additionally requires a common threshold representation and all subset bounds. FWC is exactly realizability. We first identify a regime where these conditions coincide. However, even at larger terminal count three, the scalar tests can accept a target that EC rejects. A separate obstruction distinguishes completion of one witness from existence of another, and cardinality caps introduce a third distinction.

\par\medskip
\noindent\begin{minipage}{\linewidth}
\textbf{Contributions.}\par\smallskip
\noindent\textbf{Necessity.} Two exchanges yield vertex exchange closure (EC) and scalar necessary bounds (Lemma~\ref{lem:exchange}, Theorem~\ref{thm:ec}, Corollaries~\ref{cor:LR}--\ref{cor:PC}).
\par\smallskip
\noindent\textbf{Pruning and exact reformulation.} Every realizable pair has a realization with at most $N^*$ items, where $N^*$ is the sum of the abscissae of all positive vertices of the two hulls. Pruning and rescaling yield equivalence with full witness closure (Lemma~\ref{lem:prune}, Theorem~\ref{thm:fwc}).
\par\smallskip
\noindent\textbf{Specified-witness completion.} A globally coherent EC witness may fail to admit sizes even though its target pair is realizable (Theorem~\ref{thm:trade-family}).
\par\smallskip
\noindent\textbf{Caps.} An irrelevance bound and a four-band classification separate cap dependence from failure of vertex-only capped EC (Theorems~\ref{thm:cap-large}--\ref{thm:cap-family}; Corollary~\ref{cor:cap-EC}).
\par\smallskip
\noindent\textbf{Small terminal counts.} For larger terminal count $b_w\le2$, the scalar bounds characterize realizability (Theorem~\ref{thm:bw2}).
\par\smallskip
\noindent\textbf{Scalar insufficiency.} Dominance and the scalar bounds do not characterize realizability at $b_w=3$: $L$ is strictly larger than EC (Theorem~\ref{thm:scalar-gap}, Remark~\ref{rem:scalar-gap-ec}, Corollary~\ref{cor:strict-L}).
\par\smallskip
\noindent\textbf{Constructions.} Six constructive regions supply explicit instances and expose a rational gap in the first five (Theorems~\ref{thm:AB}, \ref{thm:T5}, \ref{thm:T6}, and Appendix~\ref{app:templates}).
\par\smallskip
\noindent\textbf{Complexity.} The explicitly encoded problem lies in $\Sigma_2^p$ and is polynomial-time for fixed terminal counts; hardness is open (Theorem~\ref{thm:complexity}).
\par\smallskip
\noindent\textbf{Finite agreement.} On six specified integer-slope domains, dominance and two scalar inequalities already characterize the realizable pairs (Propositions~\ref{prop:finite}--\ref{prop:scalarfinite}).
\end{minipage}
\par\medskip

Throughout, a cap $K$ admits at most $K$ items in the larger problem and $K-1$ in the remainder; Section~\ref{sec:caps} gives the precise convention.
\paragraph{A reservation interpretation.}
Reserving a mandatory item of size $D-R$ leaves the same optional items competing at capacity $R$ instead of $D$. An advertising break with a mandatory insertion is one interpretation. Pinned virtual machines, positioning flights, guard intervals and reserved storage bays can lead to an analogous residual-capacity question only after specifying indivisible optional items with additive values and one scalar resource. These are modeling analogies, not claims that operational restrictions in those settings reduce to this model. We isolate the mathematical question of which two hulls such a shared instance can produce.

\subsection{Relation to prior work}\label{sec:related}
Gilmore and Gomory study unbounded knapsack functions of capacity \cite{GG1966}; Blair and Jeroslow characterize integer-programming value functions as the right-hand side varies \cite{BJ1982}. In inverse $0$--$1$ knapsack, Roland, Figueira and De Smet adjust profits to make a prescribed feasible packing optimal \cite{RFS2013}. Burkard and Pferschy recover a scalar parameter in a fixed parametric item set so that the optimal value equals a prescribed value \cite{BP1995}. These are adjacent inverse/value-function questions, but their inputs and quantifiers differ from prescribing two cardinality hulls while allowing the shared instance to vary.

The multiple-choice nested knapsack model of Armstrong, Sinha and Zoltners uses nested resource restrictions on a packing \cite{ASZ1982}. The structure here is one ground set evaluated at two scalar capacities. Cardinality-constrained knapsack algorithms optimize a given instance with a specified count restriction \cite{CKPP2000}; they do not by themselves reconstruct two prescribed numerical profiles. The two-part survey of Cacchiani, Iori, Locatelli and Martello provides broader context for classical and multidimensional variants \cite{CILM2022a,CILM2022b}.

Parametric-profit and parametric-weight knapsack are also relevant because they study families of optimal values as a parameter varies. Eben-Chaime gives a parametric algorithm for bicriteria knapsack \cite{EC1996}; Holzhauser and Krumke study approximation with parametric profits \cite{HK2017}, and Halman, Holzhauser and Krumke study parametric weights \cite{HKweights2017}. These forward parametric problems do not supply the shared-instance hull-pair characterization asked for here.

At the level of feasible families, threshold representation is closer. Hojny et al. survey the connections between knapsack polytopes, independence systems and matroids \cite{Hojny2020}. Our size-feasible families use a threshold constant and allow subsets of different cardinalities. They have the non-uniform threshold form studied for matroid independence systems by Giles and Kannan \cite{GK1980}, although knapsack feasibility need not define a matroid. Partida distinguishes this notion from thresholdness restricted to fixed-rank bases (which Partida calls simply thresholdness; we say uniform thresholdness for contrast), and explains the corresponding non-uniform asummability obstruction \cite{Partida2024}. In arXiv:2408.07810v2, these distinctions appear in Remarks~2.15 and~4.7. We use only the necessary equal-incidence sum argument for two thresholds, not a matroid-specific sufficiency theorem.

These works do not characterize the shared-instance pairs of cardinality hulls considered here.

\paragraph{What is not claimed.}
We do not claim that every EC witness completes, that vertex-only capped EC is sufficient, that the constructions classify all rational pairs with $b_w=3$, or that the decision problem is NP-complete.

Table~\ref{tab:notation} collects the notation used below.
\section{Setting and hull reduction}\label{sec:setting}
An instance consists of a finite set $E$ of items with rational values $a_i\ge0$, positive integer sizes $d_i$, and integer capacities $0\le R<D$. Write $G(S)=\sum_{i\in S}a_i$ and $d(S)=\sum_{i\in S}d_i$. Rational sizes can be scaled to integers together with the capacities; values remain rational. We call an item \emph{long} if $d_i>R$ and \emph{short} otherwise. Since individually $D$-infeasible items can be discarded, we assume $d_i\le D$ for every item considered in a realization.

For $\alpha\in\{w,r\}$, let $U^\alpha(n)$ be the maximum value of an admissible $n$-set under capacity $D$ or $R$, respectively, using $-\infty$ when no such set exists. Put $U^\alpha(0)=0$ and $b_\alpha=\min\argmax_n U^\alpha(n)$. The relevant hull is the upper concave hull of the points up to $b_\alpha$, evaluated at each integer count. A target $H_\alpha$ is a strictly increasing concave rational sequence on $\{0,\ldots,b_\alpha\}$ with $H_\alpha(0)=0$; the zero sequence $(0)$ is allowed. Write
\[
 h_n=H_w(n)-H_w(n-1),\qquad
 \rho_m=H_r(m)-H_r(m-1),\qquad
 \bar H_\alpha(k)=H_\alpha(\min\{k,b_\alpha\}).
\]
A positive count is a vertex if it is terminal or its successive slopes differ. However, an intermediate point on a linear run need not be attained by any set; confusing its hull value with a raw optimum changes the inverse problem. Let $\mathcal V_w,\mathcal V_r$ denote the positive vertex sets and put
\[N^*=\sum_{n\in\mathcal V_w}n+\sum_{m\in\mathcal V_r}m.\]
A pair is \emph{realizable} if a finite instance reproduces both hulls. The full-witness formulation below is an equivalent existence statement using at most $N^*$ items, rather than a third relaxation.

The hull is the relevant object because it determines the best net value at every per-item cost: for nonnegative item cost $c$, the envelope is
\[
 V^\alpha(c)=\max_n\{U^\alpha(n)-cn\}
           =\max_{0\le n\le b_\alpha}\{H_\alpha(n)-cn\}.
\]
Off-hull points do not affect the value; points after the first global maximum cannot improve it for $c\ge0$. At $c=0$ they can affect the set of maximizing counts, so only the envelope \emph{value} is asserted here. Hull-pair realizability means equality of the two specified increasing hulls, not equality of the full raw profiles.

We will repeatedly use the following equivalence: the target hull is obtained precisely when every feasible subset of count $k$ obeys $\bar H_\alpha(k)$ and each positive target vertex has an attaining set, since a hull is the least concave majorant of its raw profile and each of its vertices is a raw point. Every individually feasible item has $a_i\le \bar H_w(1)$, and every short item has $a_i\le\rho_1$, where $\rho_1=0$ if $b_r=0$. Unless stated otherwise, counts are uncapped.

\begin{table}[htbp]
\centering\small
\caption{Notation. The indices $w$ and $r$ refer to capacities $D$ and $R$, respectively.}
\label{tab:notation}
\begin{tabular}{@{}p{3.5cm}p{11.7cm}@{}}\toprule
Symbol & Meaning\\\midrule
$U^\alpha(n)$ & Raw maximum value among admissible $n$-sets; $U^\alpha(0)=0$.\\
$H_\alpha$, $b_\alpha$ & Prescribed increasing cardinality hull and its first maximizing count.\\
$h_n$, $\rho_m$ & Consecutive slopes of $H_w$ and $H_r$.\\
$\bar H_\alpha(k)$ & Flat extension $H_\alpha(\min\{k,b_\alpha\})$.\\
$\mathcal V_\alpha$ & Positive hull vertices: terminal count and strict slope changes.\\
Long / short & Size greater than $R$ / at most $R$; corresponding labels in EC.\\
$N^*$ & Sum of the abscissae of all positive vertices; a bound on the number of witness items.\\
\eqref{eq:LR}, \eqref{eq:LRb}, \eqref{eq:PC} & All-long removal, all-long pair exchange, and pair-complement scalar bounds.\\\bottomrule
\end{tabular}
\end{table}

\FloatBarrier
\paragraph{A four-item realization.}
For the items $(4,3),(3,1),(3,1),(2,1)$ at $D=5$ and $R=3$, with each ordered pair denoting \emph{(value, size)}, the two hulls are
\[
 H_w=(0,4,7,10),\qquad H_r=(0,4,6,8).
\]
The size-three item fits alone at $R$, whereas the three unit-size items fit together there; at $D$ the size-three item can join either one or both value-three items. This is a realization of two prescribed hulls by the same items, without any claim that every raw count value is prescribed; a power on an item pair below denotes its multiplicity.
\section{The three closures and necessity}\label{sec:ec}
The three conditions retain different information: scalar bounds, abstract vertex witnesses, and a common-size realization. The two exchanges below generate the first two levels.
\begin{lemma}[Two exchanges]\label{lem:exchange}
Let $X$ be $D$-feasible, $Z$ be $R$-feasible, $|X|=n$, and $|Z|=m$.
\begin{enumerate}[label=(\alph*)]
\item If $x\in X$ is long, then $(X\setminus\{x\})\cup(Z\setminus X)$ fits $D$, has count $n-1+m-|X\cap Z|$, and value $G(X)-a_x+G(Z)-G(X\cap Z)$.
\item If $x,x'\in X$ satisfy $d_x+d_{x'}>R$ and $y\in X\setminus\{x,x'\}$ is not in $Z$, then $\{y\}\cup Z$ fits $D$ and has count $m+1$ and value $a_y+G(Z)$.
\end{enumerate}
\end{lemma}
\begin{proof}
We derive both exchanges by comparing the sizes removed and inserted.
For (a), a long item cannot belong to $Z$, and the new size is at most $d(X)-d_x+d(Z)<d(X)\le D$. For (b), $d_y+d(Z)\le d_y+R<d_y+d_x+d_{x'}\le d(X)\le D$.
\end{proof}

In an abstract witness, ``the pair does not fit $R$'' is represented by absence of that pair from the declared family $\mathcal F_n$ of Definition~\ref{def:ec}; (EC5) is the second exchange, not a third rule.

Exchange closure asks for labeled attaining sets at the positive vertices of both hulls, a downward-closed declaration of which subsets of each $w$-witness are $r$-feasible, and injective identifications of short items between those sets, such that every long-removal and pair-complement exchange respects $H_w$. Item sizes and global coherence across pairs $(n,m)$ are not required. Those are questions of completion of a specified witness, not closure.
\begin{center}\small
\begin{tabular}{@{}ll@{}}\toprule
$X_n$, $Z_m$ & Labeled attaining sets for a $w$-vertex and an $r$-vertex.\\
$\mathcal F_n$ & Downward-closed family of declared $r$-feasible subsets of $X_n$.\\
$O_{n,m}$, $\pi_{n,m}$ & Overlap subset and its injective, value-preserving matching into $Z_m$.\\\bottomrule
\end{tabular}
\end{center}

A shorter definition retaining only the hull values would miss the overlap term in long removal. We therefore retain local item labels and declared feasible subsets, while deliberately not asking for sizes or for globally consistent identities. The distinction matters: satisfying these local exchanges need not make a specified witness completable.

\begin{definition}[Vertex exchange closure]\label{def:ec}
Require dominance $H_r(k)\le\bar H_w(k)$ for $0\le k\le b_r$. For each $w$-vertex $n$ introduce a labeled $n$-set $X_n$, values $x_{n,i}$, long/short labels, and a downward-closed family $\mathcal F_n$ of declared $r$-feasible subsets. Each $\mathcal F_n$ contains the empty set. Its singletons are precisely the short labels. For each $r$-vertex $m$ introduce an $m$-set $Z_m$ with values $z_{m,j}$. For each $(n,m)$ specify a partial matching from $O_{n,m}\subseteq X_n$ into $Z_m$, denoted $\pi_{n,m}$, with $O_{n,m}\in\mathcal F_n$. The witness is admissible if the following five groups hold:
\begin{enumerate}[label=\textup{(EC\arabic*)},leftmargin=4.5em]
\item $0\le x_{n,i}\le H_w(1)$, short values are at most $\rho_1$, and $\sum_i x_{n,i}=H_w(n)$. Every subset of $X_n$ obeys its larger bound; every member of $\mathcal F_n$ obeys its remainder bound.
\item $0\le z_{m,j}\le\rho_1$, $\sum_jz_{m,j}=H_r(m)$, and every subset of $Z_m$ obeys its remainder bound.
\item Each matched value agrees: $x_{n,i}=z_{m,\pi_{n,m}(i)}$ for $i\in O_{n,m}$. The matching is injective and uses only short labels.
\item For every $(n,m)$ and every long $i\in X_n$ (hence $i\notin O_{n,m}$), with $k=|O_{n,m}|$,
\[
 \bar H_w(n-1+m-k)\ge H_w(n)-x_{n,i}+H_r(m)
                   -\sum_{i'\in O_{n,m}}z_{m,\pi_{n,m}(i')}.
\]
\item For every $(n,m)$, every two-element subset $\{i,i'\}\subseteq X_n$ absent from $\mathcal F_n$, and every $y\in X_n\setminus(\{i,i'\}\cup O_{n,m})$,
\[
 \bar H_w(m+1)\ge x_{n,y}+H_r(m).
\]
\end{enumerate}
EC holds if such a witness exists. If $b_r=0$, there are no positive $r$-vertices, $\rho_1=0$, and the conditions (EC2)--(EC5) quantifying over $m$ are vacuous; dominance is the zero-sequence condition. The local subset bounds in (EC1) still apply. An EC witness is \emph{globally coherent} if one ground set contains every $X_n$ and $Z_m$ and all maps $\pi_{n,m}$ are restrictions of item identity on that set.
\end{definition}
\begin{theorem}[Necessity]\label{thm:ec}
Every realizable pair satisfies EC.
\end{theorem}
\begin{proof}
We read every component of the witness from one realizing instance.
Choose actual attaining sets at the positive vertices and use their values, labels, local $r$-feasible families, and actual intersections. Subset feasibility gives (EC1)--(EC2), item identity gives (EC3), and Lemma~\ref{lem:exchange} gives (EC4)--(EC5). Dominance follows because every $r$-feasible set is $w$-feasible.
\end{proof}
The definition has a finite mixed-integer encoding. The number of local subset constraints grows exponentially with witness cardinalities in general; no polynomial-size formulation is asserted. An exact realizing instance supplies an EC certificate by Theorem~\ref{thm:ec}; a numerical feasibility status alone is not such a certificate.

\subsection{Useful scalar corollaries}
The scalar consequences \eqref{eq:LR} and \eqref{eq:PC} suffice for the uncapped finite exclusion certificates.
\begin{corollary}[All-long bounds]\label{cor:LR}
At a $w$-vertex $n$ with $h_n>\rho_1$, EC implies
\begin{align}
 \bar H_w(n-1+m)&\ge\frac{n-1}{n}H_w(n)+H_r(m),\label{eq:LR}\\
 \bar H_w(m+1)&\ge\frac1nH_w(n)+H_r(m)\quad(n\ge3),\label{eq:LRb}
\end{align}
for $0\le m\le b_r$ in the uncapped setting.
\end{corollary}
\begin{proof}
We use the smallest permitted item value, then eliminate the unknown value by averaging.
A short member in the attaining $n$-set would leave an $(n-1)$-subset worth at least $H_w(n)-\rho_1>H_w(n-1)$, so the witness must be all long. Its cheapest value is at most its average and its dearest value at least its average. Its intersection with a remainder witness is empty, and no pair of its items belongs to its declared $r$-feasible family, since that family is downward closed and its singletons are precisely the short labels. Apply the two exchange constraints. The exchange constraints give the bounds at positive remainder vertices. At $m=0$, \eqref{eq:LR} is $h_n\le H_w(n)/n$ and \eqref{eq:LRb} is $H_w(1)\ge H_w(n)/n$; both follow from concavity, not from an exchange with an undefined $Z_0$. Concavity of the left sides and linear interpolation of the right sides, $H_r$ being affine between consecutive vertices, then extend the bounds between zero and consecutive remainder vertices.
\end{proof}

On the six domains of Appendix~\ref{app:computations}, no target is excluded by \eqref{eq:LRb} alone. We retain it as a distinct consequence of pair-complement exchange under the all-long hypothesis; its antecedent differs from that of \eqref{eq:PC}.
\begin{corollary}[Pair-complement bound]\label{cor:PC}
Suppose $n\ge3$ is a $w$-vertex and
\[
 \bar H_r(2)<H_w(n)-(n-2)H_w(1).
\]
Write $H=H_w(n)$ and $\rho_1=H_r(1)$, taking $\rho_1=0$ when $b_r=0$. Every realizable pair, and every EC-admissible pair, satisfies
\begin{equation}\label{eq:PC}
 \bar H_w(m+1)\ \ge\ H_r(m)+
 \min\left\{\frac Hn,\frac{H-\rho_1}{n-1}\right\}
 \qquad(0\le m\le b_r).
\end{equation}
\end{corollary}
\begin{proof}
We separate the two possible overlap sizes and use the same pair-complement exchange in each.
Every pair in the attaining $n$-set $X$ is worth at least $H-(n-2)H_w(1)$, so none can fit $R$. Consequently a $r$-attaining set $Z$ intersects $X$ in at most one item. If the intersection is empty, the most valuable item in $X\setminus Z$ is worth at least $H/n$. Otherwise the common item is short and worth at most $\rho_1$, so the most valuable remaining item is worth at least $(H-\rho_1)/(n-1)$. Choose such an item $y$, and choose two other members of $X$; their pair does not fit $R$. The pair-complement exchange admits $\{y\}\cup Z$ in $D$ and gives \eqref{eq:PC} at remainder vertices. At $m=0$, the bound follows from $\min\{H/n,(H-\rho_1)/(n-1)\}\le H/n\le H_w(1)$. Between zero and consecutive positive remainder vertices, the right side is linear in $m$ and the left side is concave, so the inequality extends to all intermediate counts. The same argument is internal to EC: the subset bounds exclude every pair from its declared family, the overlap has size at most one, and the corresponding exchange constraint applies.
\end{proof}
\par\noindent\textit{Remark (capped witnesses).} For a vertex-only capped witness (Section~\ref{sec:caps}), retain the bound at $r$-vertices with admissible output counts; interpolation requires both endpoints to remain available.

Figure~\ref{fig:exchange} summarizes the two exchanges, the pair-complement bound, and the balanced trade of Section~\ref{sec:trade}.
\begin{example}[Dominance is not sufficient]\label{ex:dominance}
The pair $H_w=(0,3,6)$, $H_r=(0,2,4)$ satisfies dominance. At the $w$-vertex $n=2$, $h_2=3>2=\rho_1$, but the long-removal bound at $m=2$ would require $6\ge3+4$. Thus the pair is not realizable.
\end{example}
\begin{proposition}[A flat larger hull can still be incompatible]\label{prop:flat-obstruction}
For rational $a>0$ and $a/2<r<a$, the pair
\[
 H_w=(0,a,2a,3a),\qquad H_r=(0,a,a+r,a+2r)
\]
satisfies dominance but is not realizable, uncapped or at any integer cap. Caps $K\le3$ are dimensionally excluded by $b_r=3$.
\end{proposition}
\begin{proof}
We apply the pair-complement bound at the terminal remainder count.
At $n=3$, $H_r(2)=a+r<2a=H_w(3)-H_w(1)$, so Corollary~\ref{cor:PC} applies with $H=3a$ and $\rho_1=a$. At $m=3$ it requires $3a\ge2a+2r$, contrary to $r>a/2$. This proof also applies to every admissible cap $K\ge4$ because its output count is four.
\end{proof}
The targets $(0,4,7,10),(0,3,5,7)$ and $(0,4,8,11),(0,4,6,8)$ give further examples. Their pair-complement lower bounds at count four are $31/3>10$ and $23/2>11$, respectively. In contrast, the obstruction for $(0,3,5,7),(0,1,2,3)$ uses an output of count five, so a cap of four suppresses this particular exchange.
\subsection{Full witnesses and the closure hierarchy}
\begin{lemma}[Pruning]\label{lem:prune}
Every realizable pair has a realization using at most $N^*$ items. The same statement holds with cardinality caps.
\end{lemma}
\begin{proof}
We retain only the items needed to attain the target vertices.
Retain one attaining set at each positive vertex of either hull and discard every other item. Deletion cannot increase any profile value and preserves every selected vertex value. Thus the upper hull is bounded above by, and contains all vertices of, the target hull. It equals the target. The target terminal value is still attained at its terminal vertex, and strict increase prevents an earlier attainment. The union of the retained sets has at most $N^*$ elements. Nothing in this argument changes when feasibility includes a cardinality cap.
\end{proof}
\begin{definition}[Full witness]\label{def:fwc}
A full witness consists of a ground set $E$ of at most $N^*$ items; nonnegative rational values $a_i$; vertex sets $X_n,Z_m\subseteq E$ of the specified cardinalities and values; and complete families $\mathcal D,\mathcal R\subseteq2^E$ having a \emph{common positive threshold representation}
\[
 \mathcal D=\{S:d(S)\le D\},\qquad
 \mathcal R=\{S:d(S)\le R\},\qquad 0\le R<D,
\]
with positive rational $d_i$. The vertex sets belong to the respective families and every member obeys its respective bound $G(S)\le\bar H_w(|S|)$ or $G(S)\le\bar H_r(|S|)$. With caps, impose the bounds and attainments only at admissible cardinalities. The full-witness closure $\mathrm{FWC}$ is the set of target pairs possessing such a witness.
\end{definition}
Although nested downward-closed families look like feasible families, this alone is insufficient: their representation by the \emph{same} item sizes is essential. A globally coherent interpretation of item identities is also essential; pairwise overlap matchings alone need not supply one.
\begin{theorem}\label{thm:fwc}
A target pair is realizable if and only if it belongs to $\mathrm{FWC}$.
\end{theorem}
\begin{proof}
We obtain one direction by pruning an instance and using its actual feasible families. For the converse, scale the rational threshold representation to integer sizes and capacities, leaving the values rational. The bounds keep every profile below its target hull, while the specified vertex sets attain all target vertices. Consequently both hulls are exactly the targets.
\end{proof}
\paragraph{Linear closure.}
Let $L$ be the set of target pairs satisfying dominance and all applicable instances of \eqref{eq:LR}, \eqref{eq:LRb} and \eqref{eq:PC}, with their stated hypotheses. These are explicit scalar tests; no witness is part of the definition of $L$. Every uncapped negative certificate on the six domains in Appendix~\ref{app:computations} fails \eqref{eq:LR} or \eqref{eq:PC}; Theorem~\ref{thm:scalar-gap} gives an exclusion outside those domains that passes both tests. The capped separation below additionally uses \eqref{eq:truncated}; dimensional cap exclusions are separate.

The full-witness condition is realizability, EC is its witness relaxation, and $L$ is the scalar relaxation. Theorem~\ref{thm:fwc}, Theorem~\ref{thm:ec} and Corollaries~\ref{cor:LR}--\ref{cor:PC} give
\[
 \boxed{\quad\text{realizable}=\mathrm{FWC}\ \subseteq\ \mathrm{EC}\ \subseteq\ L.\quad}
\]
Theorem~\ref{thm:scalar-gap} and Remark~\ref{rem:scalar-gap-ec} show that $L$ strictly contains EC already when $b_w=3$. Whether EC equals FWC in that regime remains open. Completion of a \emph{specified} globally coherent witness is a stronger demand than existence of some full witness for its target.
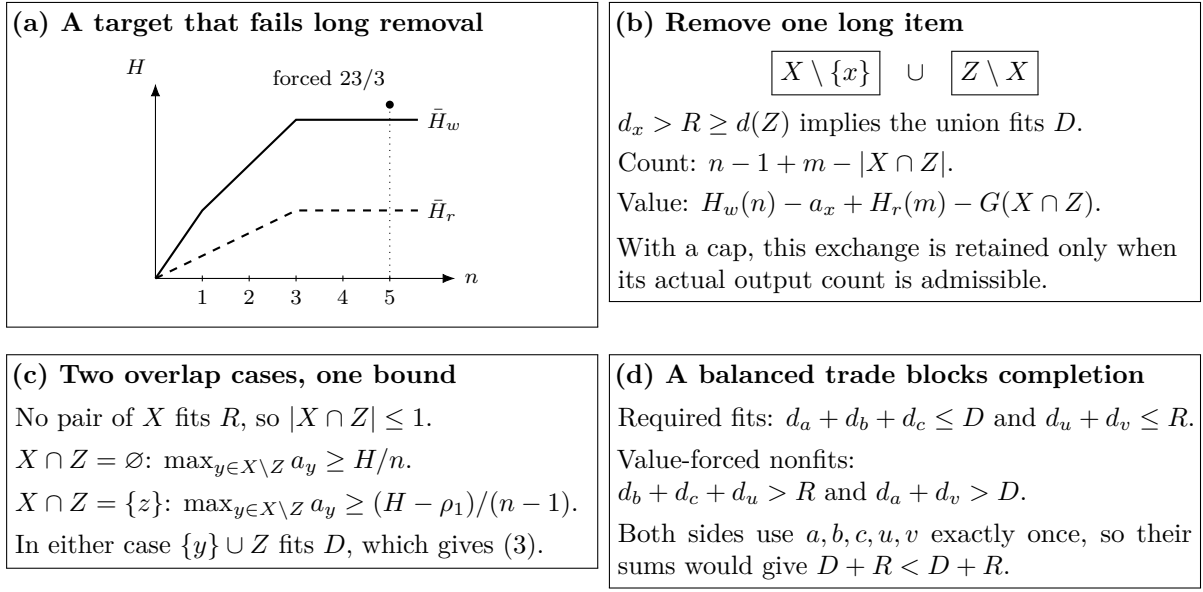
\begin{figure}[t]
\centering
\begingroup
\small
\newcommand{\panel}[2]{%
  \fbox{\begin{minipage}[t][][t]{0.46\textwidth}
    \textbf{#1}\par\vspace{4pt}#2
  \end{minipage}}}
\setlength{\tabcolsep}{2pt}
\begin{tabular}{@{}ll@{}}
\panel{(a) A target that fails long removal}{%
  \centering
  \begin{tikzpicture}[x=.62cm,y=.30cm,font=\scriptsize,>=Latex]
    \draw[->] (0,0)--(6.4,0) node[right] {$n$};
    \draw[->] (0,0)--(0,8.6) node[above left] {$H$};
    \foreach \n in {1,2,3,4,5}{\draw (\n,.13)--(\n,-.13) node[below] {\n};}
    \draw[thick] (0,0)--(1,3)--(2,5)--(3,7)--(5.6,7);
    \draw[dashed,thick] (0,0)--(1,1)--(2,2)--(3,3)--(5.6,3);
    \node[anchor=west] at (5.6,7) {$\bar H_w$};
    \node[anchor=west] at (5.6,3) {$\bar H_r$};
    \draw[dotted] (5,0)--(5,7.667);
    \fill (5,7.667) circle (1.5pt);
    \node[anchor=south east] at (5.1,7.9) {forced $23/3$};
  \end{tikzpicture}\par\vspace{2pt}}
&
\panel{(b) Remove one long item}{%
  \centering
  \fbox{$X\setminus\{x\}$}\quad$\cup$\quad\fbox{$Z\setminus X$}\par\vspace{5pt}
  \raggedright
  $d_x>R\ge d(Z)$ implies the union fits $D$.\par\vspace{3pt}
  Count: $n-1+m-|X\cap Z|$.\par\vspace{3pt}
  Value: $H_w(n)-a_x+H_r(m)-G(X\cap Z)$.\par\vspace{5pt}
  With a cap, this exchange is retained only when its actual output count is admissible.}
\\\noalign{\vskip10pt}
\panel{(c) Two overlap cases, one bound}{%
  No pair of $X$ fits $R$, so $|X\cap Z|\le1$.\par\vspace{4pt}
  $X\cap Z=\varnothing$: $\max_{y\in X\setminus Z}a_y\ge H/n$.\par\vspace{4pt}
  $X\cap Z=\{z\}$: $\max_{y\in X\setminus Z}a_y\ge(H-\rho_1)/(n-1)$.\par\vspace{4pt}
  In either case $\{y\}\cup Z$ fits $D$, which gives~\eqref{eq:PC}.}
&
\panel{(d) A balanced trade blocks completion}{%
  Required fits: $d_a+d_b+d_c\le D$ and $d_u+d_v\le R$.\par\vspace{4pt}
  Value-forced nonfits:\par $d_b+d_c+d_u>R$ and $d_a+d_v>D$.\par\vspace{4pt}
  Both sides use $a,b,c,u,v$ exactly once, so their sums would give $D+R<D+R$.}
\\
\end{tabular}
\endgroup
\space
\caption{The common-size restrictions are stronger than numerical dominance. (a) For $(0,3,5,7),(0,1,2,3)$, long removal at count five forces $23/3>7$; a cap of four suppresses that exchange. (b) Long removal accounts explicitly for overlap. (c) Pair exclusion bounds the best nonoverlapped item, whether the overlap is empty or a singleton. (d) A balanced incidence trade obstructs completion of a fixed witness while leaving target existence a separate question.}
\label{fig:exchange}
\end{figure}
\FloatBarrier
\section{Where the chain collapses}\label{sec:construct}
We begin with a regime where the scalar tests are sufficient, rather than just necessary. Beyond this regime, the subsequent constructions give a coverage map for $b_w=3$, but do not exhaust the rational targets. Unless stated otherwise, all instances here are uncapped. An instance with $N$ items remains valid at every $K\ge N+1$, which makes both caps inactive.
\begin{theorem}[Terminal count at most two]\label{thm:bw2}
For $b_w\le2$, $L=\mathrm{EC}=\mathrm{FWC}$: a pair is realizable if and only if it satisfies dominance and the all-long bounds of Corollary~\ref{cor:LR}. For $b_w=2$ the only additional restriction is
\[
 h_2>\rho_1\quad\Longrightarrow\quad
 H_r(b_r)\le\tfrac12 H_w(2).
\]
\end{theorem}
\begin{proof}
We treat the two possible terminal counts, then split according to whether the last larger slope exceeds the best short value.
Necessity follows from the exchange inequalities, using the terminal remainder count in \eqref{eq:LR}. If $b_w=0$, dominance forces $b_r=0$ and the empty instance suffices. If $b_w=1$, take the unit-size short items for $H_r$, with $R=\max\{b_r,1\}$, and a single long item of value $H_w(1)$ and size $D=R+1$. It fits with no other item, and dominance bounds every short-only set.

Suppose $b_w=2$. When $h_2\le\rho_1$, take $b_r$ unit-size items of values $\rho_j$, set $R=\max\{b_r,1\}$ and $D=b_r+5$, and add $(h_1,D)$ and $(H_w(2)-\rho_1,D-1)$. The latter item has positive value at most $h_1$ and fits with at most one short item; with the best short it attains $H_w(2)$. The two long items cannot fit together. Short-only sets obey dominance, so these items realize the targets. Otherwise take the unit-size short items, set $R=\max\{b_r,1\}$, put $s=R+1$ and $D=2s$, and add two items $(H_w(2)/2,s)$ and a solo item $(h_1,D)$. The two equal-valued items fill $D$ and attain $H_w(2)$. The solo attains $h_1$, which is at least $H_w(2)/2$ by concavity. Every other set consists of short items, or one equal-valued item with short items; its value is at most $H_w(2)/2+H_r(b_r)\le H_w(2)$. Singletons and pairs obey their respective bounds, so the hull is exactly the target. The case $b_r=0$ is included by taking no short items.
\end{proof}

For terminal count at most two, the scalar tests therefore settle the inverse problem. When $b_w=2$ and $h_2>\rho_1$, the extra restriction is just \eqref{eq:LR} at $n=2$ and the terminal remainder count; no additional named inequality is needed.

\begin{theorem}[A: free mixing; B: exclusive long items]\label{thm:AB}
A target pair satisfying dominance is realizable under either of the following conditions:
\begin{enumerate}[label=\textup{(\Alph*)}]
\item $b_r\le b_w$ and the multiset of remainder slopes is a submultiset of the larger slopes.
\item $b_r\ge b_w-1$ and $h_n\le\rho_{n-1}$ for $2\le n\le b_w$.
\end{enumerate}
\end{theorem}
\begin{proof}
We give separate size assignments for unrestricted mixing and for exclusive long items.
If $b_w=0$, dominance forces $b_r=0$; use the empty instance with $D=1,R=0$. Assume henceforth $b_w\ge1$.
For A, take $b_r$ unit-size short items of values $\rho_j$ and set $R=\max\{b_r,1\}$. For each unused occurrence in the multiset of larger slopes, take an item of that value and size $R+1$. Set $D=R+b_w(R+1)$. Every item fits together in $D$, whereas precisely the short items fit $R$, where they all fit together. Selecting the largest available values gives the two prescribed profiles.

For B, take the same unit-size short items and $R=\max\{b_r,1\}$, set $D=2b_w+b_r+1$, and add, for each $n=1,\ldots,b_w$, an item
\[
 \ell_n=\bigl(\beta_n,D-(n-1)\bigr),\qquad
 \beta_n=H_w(n)-H_r(n-1).
\]
Dominance and strict increase imply $\beta_n\ge h_n>0$. The slope hypothesis makes $\beta_n$ nonincreasing. Every $\ell_n$ is long, no two such items fit $D$, and $\ell_n$ admits at most $n-1$ unit-size items. A feasible $k$-set containing a long item therefore uses some $n\ge k$ and is worth at most
$\beta_n+H_r(k-1)\le\beta_k+H_r(k-1)=H_w(k)$.
Equality is attained by $\ell_k$ and the best $k-1$ short items, which exist because $b_r\ge b_w-1$. No set with a long item has count greater than $b_w$. Short-only sets obey dominance at every count. Hence both profiles have the required hulls.
\end{proof}

\subsection{A coverage map for terminal count three}
We use two further constructions proved in Appendix~\ref{app:templates}. T1 uses a triple, a partner and a solo item when an interval for the parameter $\lambda_1$ of Theorem~\ref{thm:T1} is nonempty. Within the cone $h_3>\rho_1$, dominance and the all-long bounds suffice when $2$ is not a $w$-vertex or $2H_r(b_r)\le h_1+h_3$ (Corollary~\ref{cor:T1-region}). T2 uses two long items and the best short within $h_3\le\rho_1$, subject to its pair-attainment conditions (Theorem~\ref{thm:T2}). They complement free mixing (A) and exclusive long items (B).

Table~\ref{tab:coverage} lists, by $w$-vertex set and $b_r$, the cells that retain residuals, with totals over all cells. ``Base'' counts the union of A, B, T1 and T2; the Applicable base column lists constructions that apply to at least one realizable target in that cell, not a universal condition on the cell. Residuals retain this definition even after another exact instance resolves them. The finite scalar characterization is stated in Proposition~\ref{prop:scalarfinite}.
\begin{center}
\small
\setlength{\tabcolsep}{4pt}\begin{longtable}{clclrrrrrr}
\caption{Coverage cells with residuals, by slope bound $B$ of the box, positive vertex set and remainder terminal count; totals are over all $16$ cells of each box. The $13$ cells per box without residuals are listed in \texttt{coverage\_by\_vertex.csv}. Targets are the realizable pairs in the dominated box; ``base'' means covered by at least one of A, B, T1 and T2. The last three columns partition the base residuals. Full row-level data also retain the dominated counts, from which the excluded counts follow by subtraction.}\label{tab:coverage}\\
\toprule
$B$ & $\operatorname{vert}^{+}(H_w)$ & $b_r$ & Applicable base & Targets & Base & Residual & T5 & T6 & Other\\
\midrule
\endfirsthead
\toprule
$B$ & $\operatorname{vert}^{+}(H_w)$ & $b_r$ & Applicable base & Targets & Base & Residual & T5 & T6 & Other\\
\midrule
\endhead
5 & $\{1,3\}$ & 3 & A, B, T1, T2 & 113 & 108 & 5 & 4 & 0 & 1\\
5 & $\{2,3\}$ & 3 & A, B, T1, T2 & 134 & 133 & 1 & 0 & 0 & 1\\
5 & $\{1,2,3\}$ & 3 & A, B, T1, T2 & 164 & 162 & 2 & 1 & 0 & 1\\
\midrule
6 & $\{1,3\}$ & 3 & A, B, T1, T2 & 252 & 238 & 14 & 8 & 0 & 6\\
6 & $\{2,3\}$ & 3 & A, B, T1, T2 & 288 & 284 & 4 & 0 & 0 & 4\\
6 & $\{1,2,3\}$ & 3 & A, B, T1, T2 & 488 & 479 & 9 & 5 & 0 & 4\\
\midrule
5 & \multicolumn{3}{l}{All $16$ cells} & 933 & 925 & 8 & 5 & 0 & 3\\
6 & \multicolumn{3}{l}{All $16$ cells} & 2086 & 2059 & 27 & 13 & 0 & 14\\
\bottomrule
\end{longtable}
\end{center}
The predicates for the base constructions are given in Theorem~\ref{thm:AB} for A and B and in Appendix~\ref{app:templates} for T1 and T2. The labels T5 and T6 match the column names of the archive; there are no constructions T3 or T4. The machine-readable version of Table~\ref{tab:coverage} retains all applicable base constructions in each cell, rather than assigning a pair to a unique construction. Thus the base count is a union, not a sum of construction counts.

Although all residuals in both displayed boxes have $b_r=3$, they do \emph{not} have one common $w$-vertex type. At slope bound five their distribution over $\{1,3\}$, $\{2,3\}$ and $\{1,2,3\}$ is $5,1,2$; at bound six it is $14,4,9$. There are no residuals of type $\{3\}$ in these boxes. T5 closes respectively five and thirteen residuals. The remaining three and fourteen have other exact instances. These statements are finite coverage facts, not a rational classification.

T6 has $b_r=2$ and, when $b<a$, $w$-vertex type $\{2,3\}$; its rational gap is distinct from the $b_r=3$ residual cells in these integer boxes. If $b=a$, its $w$-hull is linear and has vertex type $\{3\}$. Thus the table does not confine the rational problem to one vertex type or one value of $b_r$.

\subsection{Two size patterns beyond the base cover}

\begin{theorem}[T5: an isolated best short]\label{thm:T5}
Let $A\ge p\ge q\ge r>0$ be rational, and assume $p+q+r>A$. The four items
\[
 (A,3),(p,1),(q,1),(r,1),\qquad D=5,\ R=3,
\]
have raw $r$-profile $(0,A,p+q,p+q+r)$. Only its count-two point can fall below the hull. Replacing that point by its hull value, the items realize
\begin{align*}
 H_w&=(0,A,A+p,A+p+q),\\
 H_r&=\left(0,A,\max\left\{p+q,\frac{A+p+q+r}{2}\right\},p+q+r\right).
\end{align*}
\end{theorem}
\begin{proof}
We isolate the best short item and enumerate the feasible mixed types.
The size-three item fits alone in $R$ and with no other item. All three size-one items fit together in $R$. Thus its raw count profile is $(0,A,p+q,p+q+r)$. Since $A\ge p\ge q\ge r$, the first positive point lies on or above every chord from the origin to a later point. Only the count-two point may lie below its upper hull; replacing it by the larger of its raw value and the chord from counts one to three gives the displayed $H_r$. The assumption $p+q+r>A$, together with $r>0$, makes the maximum first attained at count three.

In $D$, the distinguished item fits with any two ordinary items, but the full four-set has size six. The best pair is worth $A+p$ and the best triple $A+p+q$, because $A\ge p\ge q\ge r$. Its raw profile is already concave and strictly increasing, so equals $H_w$.
\end{proof}
T5 accounts for the entries of the T5 column of Table~\ref{tab:coverage}. Its equal-first-value condition excludes the two bound-five residuals with $H_w(1)\ne H_r(1)$, and it does not cover every equal-first-value residual either.
\begin{theorem}[T6: asymmetric short sizes]\label{thm:T6}
Let rational parameters satisfy
\[
 a\ge b>\frac{2a}{3},\qquad u\ge v>0,\qquad u+v\le b,
 \qquad a+u\le2b.
\]
Then
\[
 H_w=(0,a,2a,3b),\qquad H_r=(0,u,u+v)
\]
is realized by three copies of $(b,4)$, two copies of $(a,6)$, and the two items $(u,1),(v,2)$, with $D=12$ and $R=3$.
\end{theorem}
\begin{proof}
We first check the feasible types and then locate their upper hulls.
Only the last two items fit $R$, where they fit together. In $D$, the singleton of value $a$, the pair of value $2a$, and the triple of value $3b$ give the required attainments. Every singleton is worth at most $a$, and every pair at most $2a$, because $v\le u\le b\le a$; here $u\le b$ follows from $a+u\le2b$.

For triples, two size-six items leave no room for another item. A size-six item with two size-four items does not fit. A size-six item with one size-four item and one short is worth at most $a+b+u\le3b$. Two size-four items and one short are worth at most $2b+u\le3b$. One long with both shorts is worth at most $a+u+v\le a+b<3b$, since $a<3b/2<2b$. The triple of size-four items fills the knapsack and is worth exactly $3b$.

The only feasible four-item type consists of two size-four items and both shorts: its size is $11$ and its value is $2b+u+v\le3b$. In particular, a size-six item and a size-four item admit either short separately but not both, since $6+4+1+2=13>12$. No five-set fits. Thus no later count exceeds $3b$. The inequalities $2a<3b\le3a$ make the claimed $w$-profile increasing and concave.
\end{proof}
\begin{corollary}[An infinite gap in the cover by A, B, T1, T2 and T5]\label{cor:T6-gap}
For every rational $v$ with $3/2<v\le5/3$, the pair
\[
 (0,5,10,14),\qquad (0,3,3+v)
\]
is realizable by T6 but satisfies none of the hypotheses of A, B, T1, T2, or T5.
\end{corollary}
\begin{proof}
We test the hypotheses of each preceding construction on the displayed family.
Set $(a,b,u)=(5,14/3,3)$ in T6. Construction A fails because the remainder slope $3$ is not among $5,5,4$. Construction B fails because $h_2=5>\rho_1=3$. T2 fails because $h_3=4>\rho_1$. T5 requires three remainder counts and equal first hull values. For T1, vertex $2$ requires $\lambda_1\ge5$, while its total-remainder bound requires
$\lambda_1\le14-2(3+v)=8-2v<5$. Hence its interval is empty.
\end{proof}
In particular, $v=8/5$ gives the exact rational example displayed in Table~\ref{tab:instances}. This family demonstrates a genuine gap in that cover; it is not a counterexample to EC sufficiency.

Consequently A, B, T1, T2 and T5 do not cover all rational pairs with $b_w=3$, independently of the question whether EC is sufficient. T6 supplies a realizable family in that gap.
\begin{table}[htbp]
\centering\small\setlength{\tabcolsep}{3pt}
\caption{Exact instances. Each item is (value, size); powers denote multiplicity. The last column $N$ is the number of items used, not a minimality assertion. Only the last row imposes a cap.}\label{tab:instances}
\begin{tabular}{llp{3.6cm}p{4.5cm}@{\hspace{6pt}}cr}
\toprule
Family & Role & $(H_w;H_r)$ & Items & $(D,R)$ & $N$\\
\midrule
A & Separated & $(0,5,10,12);(0,5,7)$ & $(5,1),(2,1),(5,3)$ & $(11,2)$ & 3\\
B & Overlap & $(0,5,7,8);(0,3,5)$ & $(3,1),(2,1),(5,9),(4,8),(3,7)$ & $(9,2)$ & 5\\
T1 & T1 cone & $(0,4,8,10);(0,1,2,3)$ & $(1,1)^3,(4,4),(3,5)^2,(4,10)$ & $(14,3)$ & 7\\
T2 & T2 cone & $(0,4,8,11);(0,3,5)$ & $(3,1),(2,1),(4,3),(4,4)$ & $(8,2)$ & 4\\
T5 & Isolated short & $(0,4,7,10);(0,4,6,8)$ & $(4,3),(3,1)^2,(2,1)$ & $(5,3)$ & 4\\
T6 & Template gap & $(0,5,10,14);(0,3,23/5)$ & $(14/3,4)^3,(5,6)^2,$\newline $(3,1),(8/5,2)$ & $(12,3)$ & 7\\
$K=4$ & Cap-only & $(0,3,5,7);(0,1,2,3)$ & $(3,4),(2,4)^2,(1,1)^3$ & $(12,3)$ & 6\\
\bottomrule
\end{tabular}
\end{table}
\FloatBarrier
\subsection{The scalar tests are not sufficient}\label{sec:scalar-gap}
Theorem~\ref{thm:bw2} characterizes realizability by dominance and the all-long bounds when $b_w\le2$, and Proposition~\ref{prop:scalarfinite} shows that dominance, \eqref{eq:LR} and \eqref{eq:PC} decide every target on the six integer-slope domains. It is natural to ask whether this finite agreement extends to all rational targets with $b_w=3$. It does not: we give a pair for which the all-long antecedent fails at the terminal vertex and the pair-complement bound holds with equality.

\begin{theorem}[The linear closure is strictly larger than realizability]\label{thm:scalar-gap}
The pair
\[
 H_w=(0,7,11,15),\qquad H_r=(0,4,7,10)
\]
satisfies dominance and every applicable instance of \eqref{eq:LR}, \eqref{eq:LRb} and \eqref{eq:PC}, and is not realizable.
\end{theorem}
\begin{proof}
We check the scalar tests, then split according to the overlap of two attaining triples. Write $v(x)$ for the value of an item $x$. The slopes are $h=(7,4,4)$ and $\rho=(4,3,3)$, so $\mathcal V_w=\mathcal V_r=\{1,3\}$, $\rho_1=4$ and $N^*=8$.

\emph{Membership in $L$.} Dominance is immediate. At the terminal $w$-vertex $n=3$ we have $h_3=4=\rho_1$, so the all-long antecedent fails and neither \eqref{eq:LR} nor \eqref{eq:LRb} applies there; at $n=1$, \eqref{eq:LR} reduces to dominance. The pair-complement antecedent holds, since $\Hrb(2)=7<15-7=H_w(3)-H_w(1)$, and
$\min\{H/3,(H-\rho_1)/2\}=\min\{5,11/2\}=5$ with $H=15$.
Thus \eqref{eq:PC} reads $7\ge5$, $11\ge9$, $15\ge12$ and $15\ge15$ at $m=0,1,2,3$: the last is an equality.

\emph{Unrealizability.} Suppose an instance realizes the pair. Let $T$ be a $D$-feasible $3$-set of value $15$ and $Z$ an $R$-feasible $3$-set of value $10$.
For $x\in T$ the set $T\setminus\{x\}$ is a $D$-feasible pair, so $v(x)\ge15-\Hb(2)=4$. For $z\in Z$ the set $Z\setminus\{z\}$ is an $R$-feasible pair, so $v(z)\ge10-\Hrb(2)=3$, and $z$ is an $R$-feasible singleton, so $v(z)\le\rho_1=4$.
An item of $T\cap Z$ therefore has value exactly $4$, and two such items would form an $R$-feasible pair of value $8>\Hrb(2)$. Hence $|T\cap Z|\le1$.

\emph{Case $T\cap Z=\varnothing$.} Not every member of $T$ is short: three shorts of value at least $4$ and at most $\rho_1=4$ sum to $12$, not $15$. Choose a long $x\in T$. By Lemma~\ref{lem:exchange}(a), $(T\setminus\{x\})\cup Z$ is $D$-feasible with count $5$ and value $25-v(x)$, so $\Hb(5)=15\ge25-v(x)$ and $v(x)\ge10>H_w(1)$, a contradiction.

\emph{Case $T\cap Z=\{z\}$.} Then $v(z)=4$, $T=\{z,y_1,y_2\}$ with $v(y_1)+v(y_2)=11$, and $Z=\{z,z_2,z_3\}$ with $v(z_2)+v(z_3)=6$ and both values in $[3,4]$, hence both equal to $3$.
At least one $y_i$ is long, since two shorts sum to at most $8$. For a long $y\in T$, Lemma~\ref{lem:exchange}(a) gives a $D$-feasible $4$-set $(T\setminus\{y\})\cup(Z\setminus T)$ of value $21-v(y)$, so $\Hb(4)=15\ge21-v(y)$ and $v(y)\ge6$.
Two longs would give $v(y_1)+v(y_2)\ge12>11$, so exactly one $y_i$ is long; the other is short and worth exactly $4$, and the long one is worth $7$.

Write $y_1$ for the long item. Feasibility of $Z$ gives $d(z)+d(z_2)+d(z_3)\le R$.
The pair $\{z,y_2\}$ has value $8>\Hrb(2)$, so $d(z)+d(y_2)>R$; together these imply $d(y_2)>d(z_2)+d(z_3)$.
However, feasibility of $T$ gives $d(y_1)+d(z)+d(y_2)\le D$, while the set $\{y_1,z,z_2,z_3\}$ has value $17>\Hb(4)$ and therefore satisfies $d(y_1)+d(z)+d(z_2)+d(z_3)>D$.
Together these give $d(z_2)+d(z_3)>d(y_2)$, the opposite strict inequality. Both overlap cases are impossible.
\end{proof}

\begin{remark}\label{rem:scalar-gap-ec}
The pair of Theorem~\ref{thm:scalar-gap} also fails EC, so it is not a counterexample to Question~\ref{q:ec-three}. We derive the contradiction from the local witness for $(n,m)=(3,3)$; no globally coherent witness is assumed.
\begin{enumerate}[label=\textup{(\roman*)},leftmargin=2.8em]
\item \textbf{Value ranges.} The complementary-pair bounds and the singleton bounds in (EC1)--(EC2) give
\[
 4\le x_{3,i}\le7,\quad \sum_i x_{3,i}=15,
 \qquad 3\le z_{3,j}\le4,\quad \sum_j z_{3,j}=10.
\]
At least one member of $X_3$ is long, since three shorts total at most $12$.
\item \textbf{The overlap has size one.} Put $k=|O_{3,3}|$. If $k=0$, (EC4) applied to any long value $x$ gives
\[
 15\ge25-x\quad\Longrightarrow\quad x\ge10>7,
\]
which is impossible. A matched item must be short and lie in both value ranges in (i), so its value is $4$. If $Z_3$ contains no member of value $4$, then no such matching exists, so $k=0$ and the contradiction just obtained already applies. Otherwise two matches would put a value-$8$ pair in $\mathcal F_3$, violating its remainder bound $\Hrb(2)=7$. The case $k=3$ is also excluded by the long member of $X_3$, since only shorts can be matched. Therefore $k=1$.
\item \textbf{The larger values are forced.} With $k=1$ and matched value $4$, (EC4) requires
\[
 15\ge15-x+10-4=21-x\quad\Longrightarrow\quad x\ge6
\]
for every long value $x$. Two longs and the third value would total at least $6+6+4=16$. Thus exactly one item is long, both shorts have value $4$, and the long item has value $7$.
\item \textbf{The pair-complement clause fails.} The two short $4$s form a pair of value $8$, so that pair is absent from $\mathcal F_3$. Its complementary item $y$ has value $7$ and is long, hence unmatched. Clause (EC5) has output count $3+1=4$ and requires
\[
 \Hb(4)=15\ge x_{3,y}+H_r(3)=7+10=17,
\]
which is false.
\end{enumerate}
No EC witness exists.
\end{remark}

\begin{corollary}\label{cor:strict-L}
For $b_w=3$, $L\ne\mathrm{EC}$ and $L\ne\mathrm{FWC}$ in the inclusions $\mathrm{FWC}\subseteq\mathrm{EC}\subseteq L$.
\end{corollary}
\begin{proof}
Theorem~\ref{thm:scalar-gap} and Remark~\ref{rem:scalar-gap-ec} give an element of $L$ outside both smaller classes.
\end{proof}
The pair uses a larger slope than any domain of Appendix~\ref{app:computations}, whose slopes are at most six. It lies where the scalar tests lose information: the all-long bounds are silent because $h_3=\rho_1$, and the pair-complement bound is tight. The proof uses a long item without the all-long antecedent, then a size trade that the scalar comparison does not retain. Question~\ref{q:scalar-three-new} asks for a replacement test.
\section{Witness and cap separations}\label{sec:separations}\subsection{Completion of a specified witness}
A witness without a size completion might suggest that its target is impossible. We show why that inference fails: another witness can realize the same target.\label{sec:trade}
The standard equal-incidence trade for two thresholds is the sum argument below; see \cite{GK1980} for non-uniform threshold matroids and \cite{Partida2024} for the distinction from fixed-rank thresholdness (which Partida calls simply thresholdness; we say uniform thresholdness for contrast). The sum argument is a necessary obstruction for arbitrary threshold families; it does not require that they are matroids. Its role here is to obstruct completion of a coherent EC witness whose numerical target is nevertheless realizable.
\begin{lemma}[Balanced threshold trade]\label{lem:trade}
Suppose two lists of sets have the same sum of item-incidence vectors, with at least one set on the second list. Assign each set a threshold label, $D$ or $R$, with the same number of each label on both lists. If all sets on the first list are required to fit their labeled thresholds and all sets on the second list are required not to fit, then no common size representation exists.
\end{lemma}
\begin{proof}
We sum the same item incidences on the two sides.
Summing the first list's weak inequalities bounds its total size by the common sum of labeled capacities. Summing the second list's strict inequalities puts that same total size strictly above the same sum of capacities. Equivalently, the same item-incidence vector cannot have total size both at most and greater than the same sum of capacities.
\end{proof}
\begin{theorem}[Noncompletable EC witnesses of realizable pairs]\label{thm:trade-family}
For rational parameters
\begin{equation}\label{eq:trade-parameters}
 B<A<2B,\qquad 2B-A\le C\le B,
\end{equation}
the pair
\[
 H_w=(0,A,A+B,A+2B),\qquad H_r=(0,A,A+C)
\]
is realizable and has a globally coherent EC witness that admits no size completion preserving its values and identities. In particular, these pairs are not counterexamples to existential sufficiency of EC.
\end{theorem}
\begin{proof}
We construct the non-completable witness and a different realization of its target.
Under \eqref{eq:trade-parameters}, the positive vertices are $\mathcal V_w=\{1,3\}$ and $\mathcal V_r=\{1,2\}$. Take seven \emph{distinct} items and the disjoint vertex sets
\[
 X_1=\{s\},\quad X_3=\{a,b,c\},\quad Z_1=\{t\},\quad Z_2=\{u,v\},
\]
with values $A;(A,B,B);A;(C,A)$, respectively. All seven items are labeled short, consistent with every value being at most $\rho_1=A$ and with the declared feasible pair $\{b,c\}$. In $X_3$, declare exactly the singletons and $\{b,c\}$, together with the empty set, $r$-feasible. In $X_1$, declare the singleton feasible. All overlaps are empty.

Every $w$-feasible subset bound holds: the largest pair in $X_3$ is worth $A+B$, its total is $A+2B$, and no singleton exceeds $A$. The declared remainder pair is worth $2B\le A+C$. The values in $Z_2$ satisfy the remainder bounds because $0<C\le B<A$. There are no long-removal constraints. The only pair-complement exchanges from $X_3$ use a leftover item of value $B$; they require
\[
 B+A\le A+B,\qquad B+(A+C)\le A+2B,
\]
which hold. Thus this is an EC witness, with no ambiguity from pairwise identifications.

In any completion, $X_3$ fits $D$ and $Z_2$ fits $R$, so
\begin{equation}\label{eq:trade-fit}
 d_a+d_b+d_c\le D,\qquad d_u+d_v\le R.
\end{equation}
However, $\{b,c,u\}$ is worth $2B+C>A+C$, so it cannot fit $R$, and $\{a,v\}$ is worth $2A>A+B$, so it cannot fit $D$. Therefore
\begin{equation}\label{eq:trade-nonfit}
 d_b+d_c+d_u>R,\qquad d_a+d_v>D.
\end{equation}
The sum of \eqref{eq:trade-fit} contradicts the sum of \eqref{eq:trade-nonfit}. This uses the same five item occurrences and one copy of each capacity on either side; the second list has two sets, so the nonemptiness condition in Lemma~\ref{lem:trade} holds.

Nevertheless the four items
\[
 (a_i,d_i)=(A,1),(C,1),(B,3),(B,4),\qquad D=8,\ R=2
\]
realize the target pair. Only the first two items fit $R$. In $D$, the best singleton is worth $A$, the best pair $A+B$, and the best triple $A+2B$. Every other triple is bounded by this value since $C\le B$. All four items have total size $9>8$. Hence the two profiles, and therefore their hulls, equal the displayed targets.
\end{proof}
\quantifierbox{Theorem~\ref{thm:trade-family} separates ``this witness completes'' from ``some witness completes''. The displayed target remains realizable; the theorem does not separate the classes EC and FWC.}
\paragraph{The same specified witness under a cap.}
The witness in Theorem~\ref{thm:trade-family} also satisfies the subset-enriched closure at $K=4$ (defined in Section~\ref{sec:caps}), yet has no size completion. All its items are short, so long removal is vacuous. For pair-complement exchanges from $X_3$, the remaining item has value $B$. Inserting all of $Z_2$ gives $B+A+C\le A+2B$ at count three; inserting a singleton gives at most $B+A=H_w(2)$, and inserting the empty set gives $B\le H_w(1)$. The exchanges with $Z_1$ obey the same bounds. All are cap-admissible. The two required fits and two forbidden sets in the trade have counts at most three in the remainder and at most four in the larger problem, so their contradiction persists. Nevertheless the four-item target realization remains valid at $K=4$.

\begin{example}\label{ex:trade}
Taking $(A,B,C)=(3,2,1)$ gives $H_w=(0,3,5,7)$ and $H_r=(0,3,4)$. The four inequalities give the contradiction directly.\footnote{Taking one half of each row in the positive-margin formulation of Proposition~\ref{prop:completion-lp} gives $\varepsilon\le0$; the row coefficients and rational multipliers are recorded in \texttt{completion\_counterexample.json}.} The alternative instance is $(3,1),(1,1),(2,3),(2,4)$ with $(D,R)=(8,2)$.
\end{example}
\subsection{The capped closure hierarchy}\label{sec:caps}
A cap $K\ge1$ admits at most $K$ items in the $w$-problem and $K-1$ in the $r$-problem, since the motivating mandatory item occupies one slot. Thus $b_w\le K$ and $b_r\le K-1$ are necessary. We retain these asymmetric caps throughout; imposing a common cap is a different admissibility convention.

Define $\mathrm{FWC}_K$ by full witnesses with those admissible counts. Vertex-only capped EC, denoted $\mathrm{EC}^{v}_K$, is Definition~\ref{def:ec} with bounds applied only at admissible counts and exchanges whose actual output count exceeds $K$ omitted. For the \emph{subset-enriched capped closure}, require the same dimensional admissibility $b_w\le K$, $b_r\le K-1$ and dominance, retain the same witness data and the cap-admissible bounds in (EC1)--(EC3), and apply both exchanges to every subset $Q\subseteq Z_m$. For each fixed pair $(n,m)$ define
\[
 q=|Q|,\qquad G(Q)=\sum_{j\in Q}z_{m,j},\qquad
 O_{n,m}(Q)=\{i\in O_{n,m}:\pi_{n,m}(i)\in Q\},\qquad k_Q=|O_{n,m}(Q)|.
\]
Thus $O_{n,m}(Q)=\pi_{n,m}^{-1}(Q)\cap O_{n,m}$ is the induced overlap on the $X_n$ side, and the induced matching is the restriction of $\pi_{n,m}$ to that set. The two exchange clauses are:
\begin{enumerate}[leftmargin=4.8em,label={}]
\item[\textup{(EC4$'$)}] For every long $i\in X_n$ and every $Q\subseteq Z_m$ with $n-1+q-k_Q\le K$,
\[
 \bar H_w(n-1+q-k_Q)\ge H_w(n)-x_{n,i}+G(Q)
 -\sum_{i'\in O_{n,m}(Q)}z_{m,\pi_{n,m}(i')}.
\]
\item[\textup{(EC5$'$)}] For every pair $\{i,i'\}\subseteq X_n$ absent from $\mathcal F_n$, every $Q\subseteq Z_m$ with $q+1\le K$, and every
$y\in X_n\setminus(\{i,i'\}\cup O_{n,m}(Q))$,
\[
 \bar H_w(q+1)\ge x_{n,y}+G(Q).
\]
\end{enumerate}
These clauses use the actual subset value $G(Q)$, not $H_r(q)$; both the overlap subtraction and the output count change when $Q$ changes. Taking $Q=Z_m$ recovers every retained vertex-only exchange. When $b_r=0$ there are no positive remainder witnesses, so the clauses quantifying over $m$ are vacuous as in Definition~\ref{def:ec}.

Define $L_K$ by dimensional admissibility, dominance, and \eqref{eq:LR}--\eqref{eq:PC} at $m=0$ and positive $r$-vertices only, whenever the output count is admissible. Interpolation across an omitted endpoint is not part of $L_K$.

\begin{lemma}[Capped hierarchy]\label{lem:capped-chain}
For every integer $K\ge1$,
\[
 \mathrm{FWC}_K\subseteq
 \{\text{pairs in the subset-enriched capped closure}\}
 \subseteq\mathrm{EC}^{v}_K\subseteq L_K.
\]
\end{lemma}
\begin{proof}
We use actual attaining sets for the first inclusion. Every subset $Q$ of a remainder witness fits $R$, so Lemma~\ref{lem:exchange} applies whenever its output count is admissible. Taking $O_{n,m}=X_n\cap Z_m$ with $\pi_{n,m}$ the identity on this full actual intersection gives $O_{n,m}(Q)=X_n\cap Q$; hence $y\notin O_{n,m}(Q)$ means $y\notin Q$, exactly the hypothesis in Lemma~\ref{lem:exchange}(b). For $K\ge3$, a pair absent from the actual $\mathcal F_n$ cannot fit $R$; for $K\le2$, no complementary item $y$ exists, so (EC5$\prime$) is vacuous. Keeping only $Q=Z_m$ yields vertex-only EC. The proofs of Corollaries~\ref{cor:LR} and~\ref{cor:PC}, at retained endpoints, give $L_K$; no interpolation across a dropped endpoint is used. For $K=1$, $b_r=0$ and $b_w\le1$, and an empty instance or one long item gives all four conditions directly.
\end{proof}

\begin{remark}[Removal of count truncations]\label{rem:cap-untruncated}
Theorem~\ref{thm:cap-large} identifies capped and uncapped \emph{realizability} for sufficiently large $K$. The corresponding relaxed definitions also lose their count truncations when $K\ge N^*$ and both hulls are nonzero: each local comparison has $n+m\le N^*$, so $n-1+q\le N^*-1$ for $q\le m$, and the pair-complement output count is also admissible. Remainder witness subsets have count at most $b_r\le N^*-1$. Thus each relaxed closure agrees with its uncapped counterpart; for $L_K$, interpolation between the retained remainder vertices recovers all the uncapped scalar tests. If $b_r=0$, every dimensionally admissible target has a realization with all items long, so the same target-level agreement holds. This is not an equality of the different closures with one another.
\end{remark}

\begin{theorem}[Cap irrelevance beyond the witness bound]\label{thm:cap-large}
For $K\ge\max\{1,N^*\}$, capped and uncapped realizability coincide.
\end{theorem}
\begin{proof}
We prune first and check separately the one subset that the remainder cap can exclude at equality.
The zero larger hull forces the zero remainder hull and is realized by the empty instance. If $b_r=0$, any larger target with $b_w\le K$ is realized in both models by the $b_w$ slope values as unit-size items, with $D=b_w$ and $R=0$ (use the empty instance when $b_w=0$). We may therefore assume both targets are nonzero.

Forwards, prune an uncapped realization. There are $N\le N^*\le K$ items, so the $w$-profile is unchanged. The remainder cap can delete only the full ground set, and only when $N=K$. Its required terminal vertex has count $b_r\le N^*-b_w<K$, so its maximum value and all required vertices are still attained. Deletion of the possible later full-ground-set point cannot change the target hull.

Conversely, prune a capped realization. Uncapping the larger problem changes nothing. Again only a $r$-feasible full ground set of size $N=K$ could be newly admitted. If that set fits $R$, it also fits $D$. Let $X$ be a $w$-attaining set of size $b_w<K$. Nonnegative values and the larger global bound imply $G(E)=G(X)=H_w(b_w)$. Since $X\subset E$ and $E$ fits $R$, $X$ already fits the capped remainder and attains at least the newly admitted value. Thus the new set cannot raise the remainder maximum or introduce an earlier attainment. All its smaller-cardinality profile values were unchanged.
\end{proof}
The equality case needs this argument: $K\ge N^*$ does not literally mean that the remainder cap is inactive on every subset. The simpler bound $K\ge N^*+1$ makes both caps physically inactive on a pruned ground set.
\begin{theorem}[Cap classification of the linear-remainder family]\label{thm:cap-family}
Let $a\ge b>c>0$ be rational, put $H=a+2b$, and consider
\[
 H_w=(0,a,a+b,H),\qquad H_r=(0,c,2c,3c).
\]
Then $H>3c$, and exactly one of the following four cases holds:
\begin{enumerate}[label=\textup{(\roman*)},leftmargin=3em]
\item $3c<H<9c/2$. The pair fails vertex-only capped EC at $K=4$ and is realizable at no integer cap and not uncapped.
\item $9c/2\le H<6c$. The pair satisfies vertex-only capped EC at $K=4$, but is realizable at no integer cap and not uncapped.
\item $6c\le H<9c$. The pair is realizable precisely at $K=4$ among all integer caps, and is not uncapped-realizable.
\item $H\ge9c$. The pair is realizable uncapped and at every integer cap $K\ge4$.
\end{enumerate}
Case \textup{(iii)} shows that enlarging the admissible counts can destroy realizability of a prescribed hull pair.
\end{theorem}
\begin{proof}
We establish the thresholds for $K=4$, larger caps, uncapped realizability, and vertex-only capped EC in turn. Since $a\ge b>c$, we have $H\ge3b>3c$. The terminal vertex $3$ is present in both hulls even when $a=b$.

\medskip\noindent\textbf{(A) Necessity at cap four.} Lemma~\ref{lem:truncated} (Section~5.3; its proof does not use this theorem) at $(n,m,q)=(3,3,2)$ applies because $h_3=b>c=\rho_1$. It requires
\[
 H\ge \frac23H+2c,\qquad\text{equivalently }H\ge6c.
\]

\medskip\noindent\textbf{(B) Construction at cap four.} For $H\ge6c$, use seven items
\[
 (H/3,4)^3,\quad(c,1)^3,\quad(a,12),\qquad D=12,\ R=3.
\]
Only the three unit-size items fit $R$, and they fit together, giving $H_r=(0,c,2c,3c)$ with $b_r=3\le K-1$. The solo item $(a,12)$ fits alone and with nothing else. It gives the larger count-one value $a$, since $H/3\le a$ follows from $b\le a$. The three size-four items fill $D$ and have value $H$.

The raw count-two value is $2H/3$, since $H\ge6c$ gives $H/3\ge2c\ge c$ and the solo item cannot be paired. The distinction between a raw value and a hull value is explicit:
\[
 U^w(2)=\frac{2H}{3}\ \le\ a+b
 =\frac{a+H}{2}=H_w(2).
\]
The right-hand side is the chord value between the attained points $(1,a)$ and $(3,H)$. Thus the prescribed count-two hull value does not require an attaining pair. Mixed triples have value at most $H$, again because $H/3\ge c$.

The only feasible four-sets have two size-four items and two shorts, or one size-four item and three shorts. Their values are $2H/3+2c$ and $H/3+3c$, respectively. Both are at most $H$ if and only if $H\ge6c$. Three size-four items already fill $D$, and the solo item cannot occur in a four-set. Hence the first larger maximum occurs at count three and the two hulls are the prescribed targets.

\medskip\noindent\textbf{(C) Exclusion at other caps when $H<9c$.} For $K\le3$, the required remainder terminal count three is dimensionally impossible. For $K\ge5$, suppose a realization existed. A larger attaining triple has value $H$. If it contained a short item, removing that item, worth at most $c$, would leave a pair worth at least $H-c>a+b$, since $b>c$. Thus its three items are long. Its cheapest item has value at most $H/3<3c$. Replacing it by a remainder-attaining triple, which is disjoint from the long triple, gives by Lemma~\ref{lem:exchange}(a) a feasible five-set worth more than $H$. This contradicts the flat extension of $H_w$. The argument also excludes an uncapped realization whenever $H<9c$.

\medskip\noindent\textbf{(D) Every cap $K\ge4$ and the uncapped model when $H\ge9c$.} We reuse the seven-item instance above. Beyond count four, its only feasible sets have two size-four items and all three shorts: their size is $11$ and their value is $2H/3+3c\le H$. No six-set fits, since even the three size-four items and three shorts have size $15$; the solo item cannot be paired. Thus no newly admissible set changes either hull. This one instance works uncapped and at every $K\ge4$. Conversely, \eqref{eq:LR} at $(n,m)=(3,3)$ requires $H\ge2H/3+3c$, hence $H\ge9c$ in the uncapped model.

\medskip\noindent\textbf{(E) The vertex-only threshold at cap four.} In any admissible witness, every member of $X_3$ has value at least $H-H_w(2)=b>c$: this follows by applying its subset bound to the complementary pair. Thus every member is long and downward closure gives $\mathcal F_3=\{\varnothing\}$. The overlap with $Z_3$ is empty since $\mathcal F_3=\{\varnothing\}$. Every pair in $X_3$ triggers (EC5) with the remaining member as $y$. The dearest member has value at least $H/3$, so (EC5) at count four forces $H\ge H/3+3c$, or $H\ge9c/2$.

Conversely, for $H\ge9c/2$ take disjoint all-long witnesses $X_3$ with three values $H/3$, and $X_1$ with value $a$ only when $1$ is a vertex, together with $Z_3$ of three values $c$. Let each declared local family consist only of the empty set and all overlaps be empty. If $a=b$, the larger hull is linear and there is no $X_1$. The subset bounds follow from $H/3\le a$ and $2H/3\le a+b$; the remainder bounds and dominance hold. Long removal from $X_1$, when present, produces $Z_3$ of value $3c\le H$. Long removal from $X_3$ has count five and is omitted at $K=4$. Pair-complement exchange from $X_3$ has count four and is valid exactly when $H/3+3c\le H$. This proves acceptance and completes all four cases.
\end{proof}
\begin{example}\label{ex:cap-family}
For $(a,b,c)=(3,2,1)$ the targets are $(0,3,5,7)$ and $(0,1,2,3)$. This pair is feasible at $K=4$ and infeasible at every $K\ge5$. Thus enlarging the admissible counts can destroy the realizability of a \emph{specified pair of hulls}, even though it enlarges each fixed instance's feasible families. Claims that every uncapped obstruction automatically persists for every cap are false.
\end{example}
\subsection{Separation from vertex-only capped EC}
A subset of an $r$-vertex witness can generate an admissible exchange even when exchanging the entire witness would exceed the cap.
\begin{lemma}[Truncated remainder-vertex exchange]\label{lem:truncated}
Suppose $n$ is a $w$-vertex with $h_n>\rho_1$, $m\ge1$ is a $r$-vertex, and $1\le q\le m$ satisfies $n-1+q\le K$. Every capped realization satisfies
\begin{equation}\label{eq:truncated}
 \bar H_w(n-1+q)\ge\frac{n-1}{n}H_w(n)+\frac qm H_r(m).
\end{equation}
The bound adds a restriction beyond \eqref{eq:LR} only under a cap; uncapped, it follows from \eqref{eq:LR} at $(n,q)$ and concavity of $H_r$.
\end{lemma}
\begin{proof}
We exchange only the most valuable permitted remainder items.
The $w$-attaining $n$-set is all long: removing any hypothetical short would violate the bound at count $n-1$. Its cheapest member is worth at most $H_w(n)/n$. The $q$ most valuable items of a $r$-attaining $m$-set have total value at least $(q/m)H_r(m)$, form an $R$-feasible subset, and are disjoint from the long set. Removing the cheapest long and inserting this subset gives the claimed admissible count and bound. The same argument runs inside the subset-enriched witness system: (EC1) forces $X_n$ to be all long and every overlap to be empty, and (EC4$\prime$) with $Q$ the $q$ most valuable members of $Z_m$ gives \eqref{eq:truncated}, so the bound is also necessary for the subset-enriched capped closure.
\end{proof}
\begin{corollary}[Vertex-only capped EC is not sufficient]\label{cor:cap-EC}
For rational $a\ge b>c>0$ and $9c/2\le a+2b<6c$, the pair
\[
 H_w=(0,a,a+b,a+2b),\qquad H_r=(0,c,2c,3c)
\]
satisfies $\mathrm{EC}^{v}_4$ but is realizable neither at $K=4$, nor at another integer cap, nor uncapped.
\end{corollary}
\begin{proof}
We apply the classification with its endpoint included.
This is case \textup{(ii)} of Theorem~\ref{thm:cap-family}. In particular, $a=b$ is allowed; its larger witness set has only the terminal vertex.
\end{proof}
\quantifierbox{Corollary~\ref{cor:cap-EC} concerns the vertex-only capped truncation, which drops exchanges whose output count exceeds $K$. It does not answer Question~\ref{q:ec-three}. The pairs are not realizable at $K=4$, although a truncated witness exists. This is the opposite target-existence pattern from Theorem~\ref{thm:trade-family}.}
\begin{example}
The targets $(0,4,7,10),(0,2,4,6)$ and $(0,5,8,11),(0,2,4,6)$ belong to this family. Another false positive, outside this family, $(0,5,9,13),(0,3,5,7)$, is excluded by the same lemma: $13<2(13)/3+2(7)/3=40/3$. Its archived vertex-only witness at $K=4$ has all-long $X_1=(5)$ and $X_3=(13/3,13/3,13/3)$, $Z_1=(3)$ and $Z_3=(7/3,7/3,7/3)$, empty overlaps, and declared families containing only the empty set.
\end{example}
The subset-enriched capped closure applies both exchanges to every subset of each remainder witness; by Lemma~\ref{lem:capped-chain}, the necessity of \eqref{eq:LR}--\eqref{eq:PC} and \eqref{eq:truncated} for enriched witnesses, and the exact positive certificates of Appendix~\ref{app:caps}, it coincides with realizability on the slope-five dominated domain at $K=3,\dots,6$. Beyond the particular family treated next, it remains a relaxation of target realizability.

\paragraph{Exact enriched closure on the classified family.}
For the family of Theorem~\ref{thm:cap-family}, the subset-enriched capped closure at $K=4$ is equivalent to realizability, both holding exactly when $H\ge6c$.
To see necessity directly inside the witness system, each member of $X_3$ has value at least $b>c$, so all are long, $\mathcal F_3=\{\varnothing\}$ and every overlap with $Z_3$ is empty. The two most valuable members of $Z_3$ total at least $2c$; removing the cheapest member of $X_3$, worth at most $H/3$, and inserting that pair gives an admissible count-four enriched exchange. Its bound is $H\ge2H/3+2c$, or $H\ge6c$. Sufficiency follows from the seven-item realization and Lemma~\ref{lem:capped-chain}. This is an infinite-family result about existence of a target, not completion of every specified witness, and it asserts no general sufficiency outside this family.
\section{Decidability and complexity}\label{sec:complexity}
We use an explicit encoding: the input lists every rational entry of $H_w,H_r$ in binary, together with $b_w,b_r$ and $K$ when present. Compressed vertex lists with binary-encoded abscissae define a different problem and are not treated.
\begin{proposition}[Fixed-witness completion]\label{prop:completion-lp}
Fix the values, globally consistent identities, vertex sets, long/short labels and declared $r$-feasible families of an EC witness. Its size completion is decided by a rational linear program with $|E|+2$ variables and at most exponentially many rows.
\end{proposition}
\begin{proof}
We separate required fits and nonfits with a positive common margin.
Normalize $D=1$ and maximize $\varepsilon$ over $d_i,R,\varepsilon$, with
$\varepsilon\le d_i\le1$, $0\le R\le1-\varepsilon$, and $0\le\varepsilon\le1$.
Require $d(X_n)\le1$, $d(Z_m)\le R$, and all declared local fit and nonfit relations. A required nonfit has gap at least $\varepsilon$. For every subset whose value exceeds its larger hull bound impose $d(S)\ge1+\varepsilon$; for every subset whose value exceeds its remainder bound impose $d(S)\ge R+\varepsilon$. Respect the relevant caps when forming these rows. No further subset constraints are needed: any unclassified subset with value at most its bound may fit or not fit. A completion gives a positive minimum of finitely many strict gaps and positive sizes, hence a feasible $\varepsilon>0$. Conversely a positive optimum enforces all strict relations, and the bounds and vertex attainments give a full witness. Rational feasibility gives rational sizes, which can be scaled to integers.
\end{proof}

\begin{lemma}[Small rational realizations]\label{lem:bits}
A realizable pair has a realization with at most $N^*$ items and polynomial encoding length in the explicit input. More precisely, for a fixed feasible complete fit pattern on $N$ items, the normalized sizes and thresholds can be chosen with a common denominator at most
\[
 Q_N=\left\lceil (N+2)^{(N+2)/2}\right\rceil.
\]
\end{lemma}
\begin{proof}
We bound a rational vertex separately for sizes and values.
Prune first. Fix the complete fit pattern of a realization and normalize $D=1$. Use the positive-margin size polytope from Proposition~\ref{prop:completion-lp}, but now classify every subset at each threshold. Its $N+2$ variables are bounded; every matrix entry and right-hand-side entry is $0$, $1$, or $-1$. The original finite strict representation makes a positive margin possible. The maximum margin is therefore positive at some vertex. Cramer's rule expresses the vertex coordinates with common denominator equal to the absolute determinant of an active full-rank $(N+2)$-row subsystem. Hadamard's inequality bounds this determinant by $(N+2)^{(N+2)/2}$. Scaling by it gives positive integer sizes and integer capacities with $D\le Q_N$.

For the values, fix the attaining sets and retain the subset upper bounds, attaining equalities, and $0\le a_i\le H_w(1)$. This nonempty bounded rational polytope has a rational vertex. Its coefficient matrix has entries $0,1,-1$ and at most $N$ independent active rows determine a vertex. Clearing the input denominators and applying Cramer's rule again gives polynomially many bits per value. Although the systems may have exponentially many rows, a vertex uses only linearly many independent rows. Finally $N^*\le b_w(b_w+1)/2+b_r(b_r+1)/2$, polynomial in the explicit input length. The same reasoning applies with caps by omitting size classifications that are unnecessary for cap-excluded subsets, or retaining the original complete geometric fit pattern.
\end{proof}
A small realization suggests an NP certificate, but we must also rule out every violating subset of that guessed instance. We do not know how to eliminate this universal check; the upper bound below retains it explicitly.
\begin{theorem}\label{thm:complexity}
Hull-pair realizability is decidable and belongs to $\Sigma_2^p$ under the explicit rational-sequence encoding. For fixed $b_w,b_r$ it is decidable in polynomial time in the bit length of the entries. These statements also hold with a specified cardinality cap. The fit-pattern enumeration used for fixed terminal counts has a constant doubly exponential in $N^*$; the assertion is a finiteness and bit-complexity result, not a practical algorithm.
\end{theorem}
\begin{proof}
We use a polynomial-size instance followed by one universally quantified subset.
Guess a polynomial-length instance from Lemma~\ref{lem:bits}, together with attaining sets for the target vertices. Their cardinalities, values, and feasibility are checked in polynomial time using rational arithmetic. The universal witness is one subset of the guessed ground set, encoded in at most $N^*$ bits, together with one capacity bit. Since $N^*$ is polynomial in the explicit input, a polynomial-time predicate checks this one quantified subset: it does not enumerate all subsets. It checks that feasibility implies the corresponding hull bound, respecting the cap when present. This is a polynomial-time predicate with one existential block followed by one universal block, proving membership in $\Sigma_2^p$.

For fixed terminal cardinalities, $N^*$ is a constant. Enumerate ground-set sizes up to $N^*$, complete fit patterns and attaining sets. For each discrete choice solve its rational size and value systems. There are at most $3^{2^N}$ nested complete fit patterns on $N$ labeled items, before the finite choices of attaining sets. The number of choices is a constant independent of the entry bit length, and rational linear programming is polynomial in that bit length. Hence the resulting decision algorithm is polynomial for fixed $b_w,b_r$.
\end{proof}
\paragraph{Limits of numerical search.}
Although the determinant bound proves the existence of a finite complete search, it does not make a small chosen grid complete. A model with $D=60$ searches one grid. Its infeasibility, even with $N=N^*$, does not prove unrestricted non-realizability. Similarly, a fixed positive normalized gap excludes representations with smaller gaps unless a valid universal bound is used. Floating-point feasibility must be replaced by exact reconstructed-instance checks before it is counted as a positive certificate.
\paragraph{Hardness.} Membership in NP, NP-hardness and $\Sigma_2^p$-hardness remain open; the upper bound above implies none of them. The paragraph after Question~\ref{q:complexity} records what a reduction would need.

\section{Which inclusions are strict?}\label{sec:open}
The scalar-to-EC gap is settled by Theorem~\ref{thm:scalar-gap} and Remark~\ref{rem:scalar-gap-ec}; the EC-to-FWC gap remains open. We do not know whether a different scalar test can capture the information lost by the present ones.
\begin{question}\label{q:ec-three}
Is $\mathrm{EC}=\mathrm{FWC}$ for all rational targets with $b_w=3$ in the uncapped model?
\end{question}
\begin{question}\label{q:ec-general}
Is $\mathrm{EC}=\mathrm{FWC}$ for arbitrary terminal counts in the uncapped model?
\end{question}
The balanced-trade family does not give a negative answer: another witness realizes every displayed target. Completing arbitrary coherent EC witnesses is therefore too strong a strategy for these existential questions.
\begin{question}\label{q:complexity}
For variable terminal counts, is the explicitly encoded decision problem in NP, NP-hard, or $\Sigma_2^p$-hard?
\end{question}
A hardness reduction must constrain all realizations of its target hulls. The polynomial-bit bound alone neither removes the universal subset check nor supplies a rigidity construction.
\begin{question}\label{q:enriched}
Is the subset-enriched capped closure sufficient in any infinite rational regime beyond cases already characterized by full realizability?
\end{question}
Theorem~\ref{thm:cap-family} settles every cap on the linear-remainder family, and the enriched closure is exact there at $K=4$. However, neither this result nor finite agreement on the boxes supplies sizes for a general specified witness: the capped version of Theorem~\ref{thm:trade-family} already forbids that stronger conclusion.
\begin{question}\label{q:scalar-three-new}
Is there a scalar test, valid whenever some member of a $w$-attaining set is long rather than requiring $h_n>\rho_1$, which together with dominance and \eqref{eq:PC} characterizes realizability for $b_w=3$?
\end{question}
Theorem~\ref{thm:scalar-gap} rules out the tests of Section~\ref{sec:ec} and identifies a case any replacement must cover: at a terminal vertex whose slope equals $\rho_1$, and at count four rather than count five.
Table~\ref{tab:coverage} also shows that residual template coverage cannot be confined to one vertex type. T6 exposes a rational gap at $b_r=2$ absent from these integer residual cells. These are distinctions a characterization would have to address, not evidence that another finite box would settle it.

\section*{Reproducibility}
The manuscript contains the proofs and their definitions. The deposited archive contains the rational instances, exact-subset certificates, two verification methods with different hull algorithms, and the finite claim-to-evidence map. The fixed catalog was checked in three complete executions---two on Linux \texttt{x86\_64} and one on macOS arm64. Each execution applied both exact verifiers and independently regenerated all 30 publication-facing result files; all three produced the same 7,044 decisions (5,516 positive certificates and 1,528 exact exclusions) and the same derived outputs. Parameter-wide theorems depend on the proofs; the finite assertions in Appendix~\ref{app:computations} additionally depend on the indicated archive records. Computational checks do not replace proof reading or mathematical peer review.

The authors are responsible for the mathematical claims. Large language models assisted with typesetting, some drafting revisions, finite certificates, and some adversarial pre-submission review. The exact programs verify finite claims on the declared domains; general results rely on the mathematical proofs, not on a bounded computation. These checks are not independent human proof verification.

\paragraph{Data and code availability.}
The deposited archive contains the data and code needed for the finite checks. The versioned artifact is at \href{https://doi.org/10.5281/zenodo.22823629}{doi:10.5281/zenodo.22823629}; the code repository is \url{https://github.com/methodtrace/2026-inverse-knapsack-hull-pairs}, tag \texttt{v1.0.0}. The paper identifier, the artifact DOI and the repository are distinct records.
\appendix
\section{Two three-count templates}\label{app:templates}
For this appendix let $b_w=3$, $H=H_w(3)$, $\mathcal V$ be its positive vertex set, and $\Lambda=\max\{h_3,H_r(b_r)\}$. Dominance is assumed throughout. We use T1 to separate three large values from the short items; T2 compensates for a smaller terminal slope by letting the best short participate in the attaining triple. When $b_r=0$, set $\rho_1=0$.
\begin{theorem}[T1: triple, partner and solo]\label{thm:T1}
If some rational $\lambda_1$ satisfies
\[
 \max\{H/3,\ [\rho_1,h_2]_{2\in\mathcal V}\}
 \le\lambda_1\le\min\{h_1,H_w(2)-\rho_1,H-2\Lambda\},
\]
then the pair is realizable. The bracketed entries are omitted when $2\notin\mathcal V$.
\end{theorem}
\begin{proof}
We choose the triple values first, then use sizes to permit only the desired partners.
Set $\lambda_3=\max\{\Lambda,H-2\lambda_1\}$ and $\lambda_2=H-\lambda_1-\lambda_3$. The bounds imply $\lambda_1\ge\lambda_2\ge\lambda_3\ge \Lambda>0$. Take unit-size short items of values $\rho_j$ and $R=\max\{b_r,1\}$. Add long items of values $\lambda_1,\lambda_2,\lambda_3$ and sizes $R+1,R+2,R+2$, respectively, and set $D=3R+5$. If $2\in\mathcal V$, add a partner of value $H_w(2)-\lambda_1$ and size $D-(R+1)$. If $1\in\mathcal V$, add a solo of value $h_1$ and size $D$.

The triple fills $D$. The partner fits with the first long exactly, but not with either other long. The solo fits alone. All singletons obey $h_1$: in particular the partner does because $\lambda_1\ge h_2$. For pairs, the largest pair of triple items is worth $H-\lambda_3\le H_w(2)$; a triple item with a short is worth at most $\lambda_1+\rho_1\le H_w(2)$; the partner with a short obeys the same bound because $\lambda_1\ge\rho_1$. The partner with the first long attains $H_w(2)$ when that count is a vertex. Short-only sets obey dominance.

For counts at least three, two triple items with any short items are worth at most $H-\lambda_3+H_r(b_r)\le H$. One triple item with short items is worth at most $\lambda_1+H_r(b_r)\le H$. The partner with short items is worth at most $H_w(2)-\lambda_1+H_r(b_r)\le H$, since $H_r(b_r)\le\lambda_3\le\lambda_1+h_3$. There are no other feasible mixed types. All required larger vertices are attained, and the short items give the $r$-profile.
\end{proof}
\begin{corollary}\label{cor:T1-region}
If $h_3>\rho_1$ and either $2\notin\mathcal V$ or $2H_r(b_r)\le h_1+h_3$, then the pair is realizable if and only if dominance and the all-long bounds hold.
\end{corollary}
\begin{proof}
We check every endpoint of the interval in T1.
Necessity is already proved. For sufficiency, the all-long bounds at $n=3$ imply $H_r(b_r)\le H/3$ and, by \eqref{eq:LRb} at $m=1$, $H/3+\rho_1\le H_w(2)$; when $b_r=0$ these follow directly from concavity. Hence $\Lambda\le H/3$. If $2\notin\mathcal V$, choose $\lambda_1=H/3$. Otherwise choose $\lambda_1=\max\{H/3,h_2\}$: $H/3$ alone can fail because the pair-attainment lower bound $h_2$ may be active. The term $\rho_1$ is smaller than $h_2$. Concavity gives $\lambda_1\le h_1$, while $h_2\le H_w(2)-\rho_1$ because $h_1>\rho_1$. Finally $h_2\le H-2\Lambda$ is equivalent to $2\Lambda\le h_1+h_3$, which follows from the stated assumption and $h_1\ge h_3$. Thus the interval in T1 is nonempty.
\end{proof}
\begin{theorem}[T2: two long items and a best short]\label{thm:T2}
Assume $b_r\ge1$ and $h_3\le\rho_1$. Suppose $\lambda_1\ge\lambda_2>0$ are rational, $\lambda_1+\lambda_2=H-\rho_1$, and
\[
 \lambda_1\le h_1,\qquad \lambda_1+\rho_1\le H_w(2),\qquad
 H_r(b_r)\le\lambda_2+\rho_1.
\]
If $2\in\mathcal V$, additionally require at least one of
\[
 h_3=\rho_1;\qquad \lambda_1=H_w(2)-\rho_1;\qquad
 \bigl(\lambda_1\ge\max\{h_2,\rho_1\},\ H_r(b_r)\le\lambda_1+h_3\bigr).
\]
Then the pair is realizable.
\end{theorem}
\begin{proof}
We let the best short complete the triple and use a partner only when count two must be attained.
Take the unit-size short items and $R=\max\{b_r,1\}$. Add long items $\lambda_1,\lambda_2$ of sizes $R+1,R+2$, and put $D=2R+4$. The two longs admit exactly one short. In the last displayed case, a partner of value $H_w(2)-\lambda_1$ and size $R+3$ is available; add it if needed for pair attainment. It fills $D$ with the first long and cannot fit with the second. Add a solo $(h_1,D)$ when $1\in\mathcal V$.

The two longs with the best short attain $H$. The pair of longs is bounded by $H_w(2)$ because $h_3\le\rho_1$, and a long with a short is bounded by $\lambda_1+\rho_1\le H_w(2)$. If the partner is present, its singleton value is at most $h_1$ because $\lambda_1\ge h_2$, and its value with a short is at most $H_w(2)$ because $\lambda_1\ge\rho_1$. Pair attainment at a vertex comes, respectively, from the pair of longs, the first long with the best short, or the partner with the first long.

Any single long with short items has value at most $\lambda_1+H_r(b_r)\le H$. The partner with short items has value at most $H_w(2)-\lambda_1+H_r(b_r)\le H$. Short-only sets obey dominance. Two longs admit at most one short; the partner with its compatible long admits none. Thus all $w$-vertex values are attained and no feasible set exceeds its hull bound.
\end{proof}
\section{Computations and exact finite certificates}\label{app:computations}
\subsection{Domains and the collapse of the uncapped hierarchy}
A target is specified by nonincreasing positive integer slopes. Every domain below has $b_w=3$. Each filtered domain is obtained by applying the filters stated below to the corresponding dominated domain; the dominated domains contain every pair satisfying dominance in the stated slope box, including the zero remainder hull. The filtered lists omit $b_r=0$, retain only targets satisfying \eqref{eq:LR} at each $w$-vertex $n$ with $H_w(n)>(n-1)H_w(1)+\rho_1$, and, when $H_w$ is linear with $H_w(1)=\rho_1$ and $\bar H_r(2)<2\rho_1$, require $H_r(b_r)\le2\rho_1$; these explicit filters define candidate domains only and are not applied to the dominated capped comparisons.

\begin{center}
\small
\begin{tabular}{p{6.1cm}rrrrr}
\toprule
Domain & Targets & $|L|$ & $|\mathrm{EC}|$ & $|\mathrm{FWC}|$ & Excluded\\
\midrule
Filtered: slopes $\le5$, $1\le b_r\le3$ & 924 & 898 & 898 & 898 & 26\\
Filtered: slopes $\le6$, $1\le b_r\le3$ & 2103 & 2030 & 2030 & 2030 & 73\\
Filtered: slopes $\le4$, $1\le b_r\le4$ & 472 & 449 & 449 & 449 & 23\\
\midrule
Dominated: slopes $\le5$, $0\le b_r\le3$ & 1089 & 933 & 933 & 933 & 156\\
Dominated: slopes $\le6$, $0\le b_r\le3$ & 2490 & 2086 & 2086 & 2086 & 404\\
Dominated: slopes $\le4$, $0\le b_r\le4$ & 613 & 469 & 469 & 469 & 144\\
\bottomrule
\end{tabular}
\end{center}
The cardinalities are intersections with the indicated finite domain. Every positive row has an exact instance. Every negative row has a strict rational violation of a proved bound; ``excluded'' never means a failed numerical search. A pending row would have neither an exact positive instance nor a proved exclusion. There are no pending rows.

\begin{proposition}[Certified finite agreement]\label{prop:finite}
On each of the six domains in the table, $\mathrm{FWC}=\mathrm{EC}=L$.
\end{proposition}
\begin{proof}[Finite-certificate proof]
For each dominated domain we regenerate all nonincreasing integer-slope sequences and verify equality of the resulting target set with the certificate keys. The filtered lists are checked for distinct keys and consistency with the corresponding dominated domain. Every positive instance is checked by exhaustive exact subset enumeration. Each of the $26$, $73$, $23$, $156$, $404$ and $144$ exclusions violates \eqref{eq:LR} or \eqref{eq:PC}. Thus the rejected rows lie outside $L$, and necessity places every realized row in EC and $L$. These inclusions certify the displayed equalities; numerical feasibility statuses are not needed to infer them.
\end{proof}

\begin{proposition}[Scalar characterization on the boxes]\label{prop:scalarfinite}
On the stated integer-slope domains, a pair is realizable if and only if it satisfies dominance, \eqref{eq:LR} and \eqref{eq:PC}, with the antecedents of those inequalities respected. No separate check of \eqref{eq:LRb} or an exchange-witness optimization problem is needed for this finite decision.
\end{proposition}
\begin{proof}[Finite-certificate proof]
Necessity holds for every realizable pair. For sufficiency within these domains, the row-level catalog \texttt{exact\_exclusion\_certificates.json} gives a violation of \eqref{eq:LR} or \eqref{eq:PC} for every rejected target, and an exact realization is supplied for every remaining target. The verifier checks both directions for every key, not merely the totals. In its additional checks it also evaluates \eqref{eq:truncated}. In the uncapped regime that form at $(n,m,q)$ follows from \eqref{eq:LR} at $(n,q)$, since concavity gives $H_r(q)\ge(q/m)H_r(m)$. Hence those additional checks do not strengthen the uncapped scalar criterion asserted here.
\end{proof}
Neither proposition asserts sufficiency for all rational pairs.

The equivalence is exact only on the stated domains. It does not extend to all rational targets: Theorem~\ref{thm:scalar-gap} supplies a scalar-admissible, unrealizable pair whose largest slope is seven, outside all six domains. The diagnostic for that pair checks its scalar slacks and the two size chains; its bounded size enumeration is not the proof of impossibility.

\subsection{Coverage and the last residuals}
Table~\ref{tab:coverage} describes the constructive coverage of the two dominated boxes with $b_r\le3$. Of the $933$ and $2086$ realizable targets, respectively, A, B, T1 and T2 cover $925$ and $2059$. All $8$ and $27$ residuals have $b_r=3$, but they are not concentrated in a single $w$-vertex type. T5 covers $5$ and $13$ of these residuals. Every other residual has an existing exact instance; these are not additional parameter-wide construction theorems.

For example, the target
\[
 H_w=(0,6,12,15),\qquad H_r=(0,2,4,5)
\]
has the eight-item realization
\[
 (6,6)^2,\ (5,4)^3,\ (2,1)^2,\ (1,1),\qquad D=12,\quad R=3.
\]
The two six-valued items attain count two, and three five-valued items attain count three. Two six-valued items fill $D$; one six-valued item and two five-valued items do not fit. One six-valued and one five-valued item admit at most two unit shorts, of total value at most four. Two five-valued items admit all three shorts. These cases bound every feasible set by the target hull.

\subsection{The capped hierarchy and independent item censuses}\label{app:caps}
The four capped comparisons use all $1089$ dominated slope-five targets, including the $35$ zero-remainder targets. The scalar class $L_K$ is the vertex-restricted class defined in Section~\ref{sec:caps}; ``enriched'' denotes the subset-enriched capped closure.
\begin{center}
\begin{tabular}{crrrrrr}
\toprule
$K$ & Targets & $|L_K|$ & $|\mathrm{EC}^{\mathrm v}_K|$ & Enriched & $|\mathrm{FWC}_K|$ & Excluded\\
\midrule
3 & 1089 & 501 & 501 & 501 & 501 & 588\\
4 & 1089 & 969 & 969 & 959 & 959 & 130\\
5 & 1089 & 933 & 933 & 933 & 933 & 156\\
6 & 1089 & 933 & 933 & 933 & 933 & 156\\
\bottomrule
\end{tabular}
\end{center}
On these same $1089$ targets, exactly $26$ pairs are realizable at $K=4$ but not uncapped, and none is realizable uncapped but not at $K=4$. One of them, $H_w=(0,3,5,7)$ and $H_r=(0,1,2,3)$, belongs to the family of Theorem~\ref{thm:cap-family} and is displayed in Example~\ref{ex:cap-family}. The target-by-target comparison is recorded in \texttt{cap\_four\_vs\_uncapped.json}.

Positive instances and necessary exclusions certify the enriched and FWC columns exactly. The ten vertex-only false positives at $K=4$ also have explicit exact witnesses, checked in \texttt{vertex\_cap\_witnesses.json}; scalar rejections exclude all other negative rows from vertex-only EC. Thus this table does not promote numerical solver statuses to proofs. It shows finite equality of the enriched closure and FWC, not general capped sufficiency.

At $K=3$, $583$ exclusions are dimensional ($b_r=3>K-1$) and $5$ are scalar. The ten vertex-only false positives at $K=4$ are rejected by the truncated bound \eqref{eq:truncated}. Item counts of the positive certificates depend on the witnesses the solver selected, are recorded in the archive, and do not establish minimum instance sizes.

The linear-remainder family provides a separate theorem regression. The three dominated domains contain, respectively, $20$, $35$ and $10$ integer members with $a\ge b>c>0$. The script \texttt{check\_family\_classification.py} compares each recorded uncapped disposition with $H\ge9c$, checks the applicable capped predictions on the slope-five domain, and verifies the constructive branches exactly. There are no mismatches. Twelve additional rational examples include $a=b$ and the three threshold endpoints. These overlapping-domain checks are not pooled as independent observations. The sampled-cap predictions are regression evidence; the assertion for every $K\ge4$ in case~(iv) is proved by Theorem~\ref{thm:cap-family}.

The two item-level censuses enumerate multisets, not random samples: every multiset of one to the stated number of items with integer values from $0$ to the maximum value and integer sizes from $1$ to the maximum size, at every pair of integer capacities $0\le R<D$ up to the stated bound on $D$. With maximum value $4$, maximum size $5$, at most four items and $D\le8$, the recorded census has $844$ distinct hull pairs. With maximum value $3$, maximum size $4$, at most five items and $D\le6$, it has $613$. The numerical EC checks recorded in the census files rejected none of those realized pairs and recorded no unknown statuses. The default verification re-executes both censuses by exact enumeration and rechecks every representative instance exactly.

\subsection{Verification and implementation}
The default verification runs two catalog checks, \texttt{tools/verify\_catalog\_a.py} (combinations and all chords) and \texttt{tools/verify\_catalog\_b.py} (bitmasks and monotone-chain hulls), and then regenerates the result files. Within the regeneration, two exact positive checkers use different algorithms: \texttt{core.py} enumerates subsets by bitmasks and constructs hulls by a monotone-chain routine, and \texttt{verify\_certificates.py} enumerates combinations and maximizes over all chords between raw-profile points. Both use exact rational arithmetic; \texttt{verify\_certificates.py} does not import the profile or hull functions of \texttt{core.py}. It also checks the negative inequalities, the four-row completion dual, all domain keys and the ten explicit capped witnesses. The tabulation script computes Table~\ref{tab:coverage} from the certified targets and the construction predicates; it performs no domain search.

The exploratory searches that produced the candidate instances used integer sizes with $D=60$ and a HiGHS-backed mixed-integer solver. Solver statuses are not evidence here: every accepted positive is reconstructed exactly and verified by exhaustive subset enumeration, and the certificate path in the deposited archive needs only the Python standard library and invokes no numerical solver. A positive target can admit several exact realizations, so independent solver runs need not select the same witness. Reproducibility of the finite claims means exact verification of the deposited catalog and deterministic regeneration of the reported classifications and tables, not byte-identical rediscovery by the optimizer.

Thirty displayed or theorem-related instances were recertified on the integer grid $D=60$, preserving the values of their source certificates and every labeled subset's two-threshold fit pattern. All thirty recertifications are checked exactly. Repeated targets with different constructions remain separate checks. Numerical examples and finite certificates do not substitute for the parameter-wide proofs.

\end{document}